\documentclass[a4paper,11pt,reqno]{amsart}
\usepackage{amsmath,amssymb,amsthm} 
\usepackage{graphics} 

\usepackage[english]{babel}

\usepackage{pgf}
\usepackage{tikz,wrapfig}
\usetikzlibrary{arrows}

\usepackage[colorlinks = true]{hyperref}
\usepackage{epsfig}

\usepackage{textcomp} 
\usepackage{amsmath,amssymb}
\usepackage{mathrsfs}
\usepackage{amsthm}
\usepackage{graphicx,float}
\usepackage{caption}
\usepackage{subfig}
\usepackage{listings}
\usepackage{hyperref}
\usepackage{mathtools}

\numberwithin{equation}{section}

\theoremstyle{plain}
\newtheorem{theorem}{Theorem}[section]
\newtheorem{proposition}[theorem]{Proposition}
\newtheorem{lemma}[theorem]{Lemma}
\newtheorem{corollary}[theorem]{Corollary}

\theoremstyle{definition}

\newtheorem{definition}[theorem]{Definition}

\theoremstyle{remark}
\newtheorem{remark}[theorem]{Remark}

\newcommand{\T}{\mathbb{T}}
\newcommand{\R}{\mathbb{R}}
\newcommand{\Z}{\mathbb{Z}}
\newcommand{\dT}{\mathsf{d}_{\T^d}}
\renewcommand{\d}{\mathrm{d}}
\newcommand{\pr}{\operatorname{pr}}

\newcommand{\norm}[1]{\left\lVert#1\right\rVert}
\newcommand{\bra}[1]{\left ( #1\right )}
\newcommand{\braq}[1]{\left [ #1\right ]}
\newcommand{\brag}[1]{\left \{ #1\right \}}

\renewcommand{\epsilon}{\varepsilon}
\renewcommand{\phi}{\varphi}
\newcommand{\abs}[1]{\left | #1\right |}

\newcommand{\Ex}{\mathbb{E}}

\DeclareMathOperator*{\Lip}{Lip}

\DeclareMathOperator*{\dist}{\mathsf{d}}
\DeclarePairedDelimiter\floor{\lfloor}{\rfloor}

\title[Wasserstein Asymptotics for fBm on a Flat Torus]{Wasserstein Asymptotics for the Empirical Measure of Fractional Brownian Motion on a Flat Torus}

\thanks{M.H. and F.M. are supported by the Deutsche Forschungsgemeinschaft (DFG, German Research Foundation) through the SPP 2265 {\it Random Geometric Systems}. M.H. and F.M. have been funded by the Deutsche Forschungsgemeinschaft (DFG, German Research Foundation) under Germany's Excellence Strategy EXC 2044 -390685587, Mathematics M\"unster: Dynamics--Geometry--Structure. D.T. is a member of the Gnampa INdAM group. }
\author[M. Huesmann]{Martin Huesmann}
\address{M.H.: Universit\"at M\"unster, Germany}
\email{martin.huesmann@uni-muenster.de}
\author[F. Mattesini]{Francesco Mattesini}
\address{F.M.: Universit\"at M\"unster \& MPI Leipzig, Germany   }
\email{francesco.mattesini@uni-muenster.de}
\author[D. Trevisan]{Dario Trevisan}
\address{D.T.: Dipartimento di Matematica, Università degli Studi di Pisa, 56125 Pisa, Italy  }
\email{dario.trevisan@unipi.it}
\date{}

\begin{document}

\begin{abstract} 
We establish asymptotic upper and lower bounds for the  Wasserstein distance of any order $p\ge 1$ between the empirical measure of a fractional Brownian motion on a flat torus and the uniform Lebesgue measure. Our inequalities reveal an interesting interaction between the Hurst index $H$ and the dimension $d$ of the state space, with a ``phase-transition'' in the rates when $d=2+1/H$, akin to the Ajtai-Koml\'os-Tusn\'ady theorem for the optimal matching of i.i.d.\ points in two-dimensions. Our proof couples PDE's and probabilistic techniques, and also yields a similar result for 
discrete-time approximations of the process, as well as a lower bound for the same problem on $\R^d$.
\end{abstract} 

\maketitle 

\section{Introduction}



The occupation measure of a process $X = (X_t)_{t \ge 0}$ up to a time $T$ is defined as
$$ \mu^X_T = \sum_{0 \le t \le T} \delta_{X_t}, \quad \text{or} \quad \mu^X_T = \int_0^T \delta_{X_t} dt,$$
depending on whether $X$ is a discrete- or continuous-time process. If renormalized to a probability measure it is also known as the empirical measure of $X$. It is a fundamental object describing the distribution of the process trajectory on the state space since, given any measurable set $A$ in the state space of the process, one has
$$ \mu^X_T(A) = | \brag{ t \in [0,T] \, : X_t \in A }|$$
where $|\cdot|$ denotes the number of elements (in the discrete-time case) or Lebesgue measure (in the continuous-time case).  From the simplest case of a discrete-time process consisting of i.i.d.\ random variables, to that of dependent variables, in particular for Markov processes, the occupation measure has many applications, ranging from non-parametric statistics, to Monte Carlo integration and mean field theory. Under natural assumptions on $X$, such as stationarity and ergodicity, limit theorems can be established for the empirical measure as  $T \to \infty$. It is an interesting and often challenging question to determine how fast convergence occurs in terms of a given metric on the space of measures, thus complementing the qualitative convergence with useful quantitative bounds, that may reveal otherwise hidden features, e.g.\ the role of dimensionality of the state space, or regularity of the process trajectories.

\medskip
If the state space of the process $X$ is naturally endowed with a distance, a natural family of metrics between measures $\mu$, $\nu$ is provided by the so-called Wasserstein distances of order $p$,  $W_p(\mu, \nu)$, defined in terms of the  optimal transport problem with cost given by the $p$-th power of the distance (where $p\ge 1$ is a chosen parameter). Also called earth mover's distance, $W_p(\mu, \nu)$ is then the minimum total cost of turning $\mu$, thought as a distribution of mass, into $\nu$, by physically moving it in the ambient space. The case $p=1$ is the classical one, dating back to Monge and Kantorovich, and $W_1(\mu, \nu)$ has also a natural dual formulation in terms of Lipschiz functions:
$$ W_1 (\mu, \nu) = \sup _{\Lip(f) \le 1 }\brag{ \int f \,\d \mu - \int f \,\d \nu},$$
but in recent years other choices, in particular $p=2$, have been subject of intense investigations \cite{ambrosio2005gradient, villani2009optimal, santambrogio2015optimal}.

\medskip
The aim of this work is to establish precise results for the empirical measure of the fractional Brownian motion (fBm), which constitutes a fundamental example of a continuous time process that is not Markov (except when it coincides with usual Brownian motion).  From an applied perspective, fBm is widely used to model real-world phenomena, in particular those exhibiting long-range dependence, and has been used in various fields, from biology, to telecommunication engineering and finance \cite{decreusefond1998fractional, jumarie2006new, fallahgoul2016fractional}. Our primary motivation comes from the fact that the vast and growing literature on large-time asymptotic results for empirical measures of stochastic processes focuses mostly on the case of i.i.d.\ random variables or Markov processes, see e.g.\ \cite{fournier2015rate, riekert2021wasserstein} for Markov chains, and examples outside this class are quite rare, see e.g.\  \cite{jalowy2021wasserstein} for an interesting application to random matrix theory. Still, fBm is a Gaussian process hence it is amenable for many explicit computations, a fact that stimulated the development of novel approaches to stochastic analysis \cite{biagini_stochastic_2008, nourdin2012selected}.

\subsection*{Main result}
Our main theorem can be stated as follows:

\begin{theorem}\label{thm:main-Td}
Let $B^H = \bra{B_t^H}_{t \ge 0}$ be a fractional Brownian motion with Hurst index $H \in \bra{0,1}$ with values on a $d$-dimensional torus $\mathbb{T}^d$. Then, for every $p\ge 1$, as $T \rightarrow \infty$,
\begin{equation}\label{eq:rates-main-theorem}
 \Ex \braq{W_p^p \bra{ \frac 1 T \int_0^T \delta_{B^H_s} ds , \mathcal{L}^d_{\mathbb{T}^d}}} \sim  \footnote{We use the notation $A\lesssim B$ if there exists a global constant $C>0$, which may only depend on $d,p$ and $H$, such that $A \le C B$. We write $A \sim B$ if both $A \lesssim B$ and $B \lesssim A$.} \begin{cases}
T^{-p/2} & \text{if}\ d < \frac{1}{H} + 2, \\
\bra{\log T/T}^{p/2} & \text{if}\ d = \frac{1}{H} + 2, \\
T^{-\frac{p}{d-1/H} } & \text{if}\ d > \frac{1}{H} + 2,
\end{cases} 
\end{equation}
where $\mathcal{L}^d_{\mathbb{T}^d}$ denote the uniform probability (Lebesgue measure) on $\T^d$.
\end{theorem}

All the definitions (of fBm with Hurst index  $H$ and Wasserstein distance $W_p$) are precisely recalled in Section~\ref{sec:notation}. Informally, the Hurst index $H$ measures the regularity of the process trajectories (which are slightly less than $H$-H\"older continuous). A $d$-dimensional fBm is simply given by  $d$-independent copies of a one-dimensional fBm. We are currently not able to deal with $\R^d$ as a state space (because of lack of compactness), hence we technically restrict to a torus $\T^d = \R^d/\mathbb{Z}^d$ by projecting the process from $\R^d$. However, a straightforward consequence is an asymptotic lower bound in $\R^d$, Corollary \ref{cor:two-paths}, whose sharpness however is presently not clear to us.

\medskip
The asymptotic rates in \eqref{eq:rates-main-theorem} depend on the values of $d$ and $H$, with a ``phase transition'' at the critical dimension $1/H+2$: if the dimension $d$ is smaller, then the rate is akin to that of a central limit case (i.e., $\sim T^{-1/2}$ for $W_p$), while if $d$ is larger, then the rate is dimension dependent and actually suffers from a curse of dimensionality, since as $d \to \infty$ convergence is slower (keeping $p$ and $H$ fixed). This may suggest the use of small values of $H$ to explore with a fBm a large dimensional manifold, although it may be impractical from a computational point of view. The precise asymptotics, i.e. $\sim  T^{-1/\bra{d-1/H}}$ can be also intuitively interpreted with the following heuristics. Since the trajectories of fBm are (almost) $H$-H\"older continuous, then the support of the occupation measure will be roughly $1/H$-dimensional. One may then ask what is the optimal way to choose a measure with total mass $T$ supported on a $1/H$-dimensional set, to minimize its Wasserstein distance from $T \mathcal{L}^d_{\T^d}$ -- this is a generalization of the usual quantization problem \cite{graf2007foundations} where the measure must be supported on $0$-dimensional sets.  If $1/H = h$ is integer, the Hausdorff measure on the union of  $[0,1]^h \times \brag{z_i}$, with $z_i \in [0,1]^{d-h}$ in a grid with $\sim T$ points, so that $|z_i - z_j| \lesssim T^{-1/(d-h)}$, gives a feasible choice with Wasserstein distance $\lesssim T \cdot T^{-1/(d-h)}$ which coincides with the rate in \eqref{eq:rates-main-theorem} for $d>1/H$. We conjecture that, although the above construction does not provide a minimizer, the asymptotic rate for the minimization problem is indeed $ T \cdot T^{-1/(d-1/H)}$, hence the trajectories of fBm are efficient from the asymptotic perspective, although quite different from the construction above or possibly the actual minimizers.

\medskip
In the case $H=1/2$, so that fBm reduces to usual Brownian motion, our result is a special case of  those obtained by F.-Y.~Wang and collaborators \cite{wang2019limit, wang2021convergence, wang2021precise,wang2022wasserstein} for general diffusion processes on compact and even non-compact Riemannian manifolds. In the compact case, the critical dimension is always $4=2+1/H$, in line with our result.

\medskip
Even more interestingly, informally in the limit $H\to \infty$ (this is of course non rigorous since $H\in (0,1)$) the trajectories, whose dimension is $1/H$, reduce to independent points. Then, our result precisely recovers the asymptotic behaviour for $d>1$ and critical dimension $d=2$ of the empirical process of i.i.d.\ points on the torus $\T^2$ as first established by Ajtai, Koml\'os and Tusn\'ady \cite{ajtai1984optimal} for the matching problem in $[0,1]^2$ (but their argument would work also on $\T^2$). It would be fascinating to set this limit on more rigorous grounds.

\subsection*{Comments on proof technique}
The overall strategy is based on the recent PDE approach to the bipartite matching problem \cite{ambrosio2019pde}, later simplified on the torus $\T^d$ in \cite{bobkov2021simple}, which made rigorous some challenging predictions from the statistical physics literature \cite{caracciolo2014scaling}. Similar techniques, in particular related to upper bounds, have been also independently employed in the literature, to study empirical measures of random walks on groups \cite{borda2021berry, borda2021equidistribution, borda2021empirical},  Kronecker sequences on the torus \cite{steinerberger2021wasserstein} or with applications to minimal Green energy problems on compact manifolds \cite{steinerberger2021green}.

\medskip
The main idea is to replace $W_p(\mu, \nu)$ with a negative Sobolev norm $\norm{ \nabla \Delta^{-1} (\mu - \nu)}_{L^p}$, which is indeed quite close to it if both $\mu$ and $\nu$ have nice densities with respect to Lebesgue measure. While $\mu = \mathscr{L}^d$ gives no problem, for the empirical measure, which is singular, we introduce a smoothing step, using the heat semigroup with a small time parameter to be carefully tuned. Finally, we estimate the negative Sobolev norms using a Fourier expansion: this is where our approach becomes less straightforward, since we need to take into considerations the probabilistic correlations between the various coefficients, while in the classical matching between i.i.d.\ points these are almost independent. Although it is never made explicit, the underlying difference is that in the CLT scaling one has convergence in law (in the space of distributions)
$$ \frac{1}{\sqrt{T}} \bra{ \int_0^T \delta_{B_s^H} ds - T } \to \Xi$$
towards a coloured noise on $\T^d$, while in the i.i.d.\ case $\Xi$ is simply white noise (which would imply independent Fourier coefficients). The same problem would appear already in the Brownian case, $H=1/2$, but was overcome by Wang using the Markov property and bounds for the heat kernel. We overcome instead this difficulty by relying on  the \emph{local  non-determinism} of fBm \cite{xiao2006properties, galeati2020prevalence}, yielding a lower bound on the covariance matrix of the time increments of the process which in turn appears in the estimate for mixed moments of the Fourier coefficients.

\medskip
The overall approach that we employ is quite robust and can be used to establish other asymptotic bounds, e.g.\ when the continuous-time process is replaced with discrete-time approximations, e.g., with a time-step $\tau = T^{-\alpha}$, for some $\alpha>0$. This is made explicit in Theorem~\ref{thm:main-discrete} and interestingly gives a more complex picture, with rates depending also on $\alpha$. This may be relevant for applications, since numerical simulations of Brownian motion can be exactly performed on a finite grid in an interval $[0,T]$ (spectral methods \cite{dieker2003spectral} provide an alternative for which it would be interesting to establish similar results).

\subsection*{Open questions}

In view of our results, several open questions arise: 

\begin{enumerate}
\item The most relevant also for applications is whether the same bounds we have on $\T^d$ are valid on $\R^d$, i.e., complementing Corollary~\ref{cor:two-paths} with upper bounds. The closest problems are in the i.i.d.\ case, for matching of Gaussian samples, where sharp bounds are known if $p<d$ and for the critical $p=d=2$ where rates are different than the compact (torus or square) case \cite{ledoux2017optimal, ledoux2019optimal, talagrand2018scaling}; in the Markov diffusion case, Wang \cite{wang2022wasserstein} established (sharp) upper bounds for non-compact manifolds with a Gibbs measure $e^{-V}$ with a potential $V$ growing sufficiently fast, and here sharp rates, for $p=2$, are only known if the dimension is not too large (and in general $d\le 4$).

\item  A second question is whether the asymptotic upper and lower bounds in \eqref{eq:rates-main-theorem} can be made precise, showing existence of a limit, when appropriately renormalizing the expected Wasserstein distance. Again, comparison with the i.i.d.\ case (e.g.\ on $[0,1]^d$ \cite{goldman2020convergence}) and Markov diffusion cases (on compact manifolds) \cite{wang2019limit} suggest that the limit should always exist, although we should point out that in the i.i.d.\ case it is still an open question in the critical dimension $d=2$ and $p\neq 2$, while for Markov diffusions only the case $p=2$, $d <4$ is settled.

\item A relevant case in which our PDE techniques do not seem to apply (already for the matching problem) is that of $p = \infty$. This is also closely related to the cover time problem, i.e., the time it takes for the process to be uniformly close to any point in the state space, which was completely settled for Brownian motion in \cite{dembo2004cover}.

\item We finally point out that our results describe only the asymptotic behaviour of the Wasserstein distance in expectation, but stronger convergence, e.g.\ almost sure, should hold. For the i.i.d.\ case, this is almost without effort established using standard concentration of measure arguments, at least in the regime $p<d$. For Markov chains, i.e., discrete time, one may use Marton's arguments \cite{marton1996measure}, but already for diffusion processes concentration appears to be a challenging question, that we leave for future investigations.

\end{enumerate}

\subsection*{Structure of the paper} In Section~\ref{sec:notation} we introduce the notation and basic facts for fBm, Wasserstein distance and Fourier analysis on $\T^d$. Section~\ref{sec:thm:main} is devoted to the proof of our main result, Theorem~\ref{thm:main-Td} and contains most of the technical material. Section~\ref{sec:corollary} shows Corollary~\ref{cor:two-paths} yielding a lower bound for the expected distance between two independent fBm's on $\R^d$. Finally, in Section~\ref{sec:thm:discrete} we show how to modify our derivations to obtain a variant of Theorem~\ref{thm:main-Td} where fBm is sampled at discrete times.

\subsection*{Acknowledgements} The authors thank L.~Galeati for useful discussions and suggesting the use of the local non-determinism of fBm, which greatly simplified our initial approach. 

\section{Notation and preliminary results} \label{sec:notation}

We write throughout $\T^d = \R^d/\Z^d$ for the $d$-dimensional flat torus, i.e., endowed with the distance
\[
\dT \bra{x,y} = \min_{k \in \Z^d} \abs{x-y-k}, \quad \text{for $x,y \in \T^d$,} 
\]
where $|u| = \sqrt{ u \cdot u}$ denotes the Euclidean norm of $u \in \R^d$. We often conveniently identify functions and measures on $\T^d$ with their periodic lift to $\R^d$. We write $|A|$ for the Lebesgue measure of a Borel set $A \subseteq \R^d$ or $A \subseteq \T^d$, and $\int_A f$ for the Lebesgue integral of $f$ on $A$. When a measure $\mu$ is absolutely continuous with respect to the Lebesgue measure (if on $\R^d$ or $\T^d$) or the counting measure (if on a countable space, such as $\Z^d$) we tacitly identify it with its density. We write $\delta_0$ for the Dirac measure with unit mass at $0$, and often treat it as a function (which is fully rigorous on a discrete space, otherwise it means that we restrict Lebesgue integration accordingly). 

\subsection{Fourier Analysis} We first recall some basic facts of Fourier analysis on $\T^d$  that will be used below. Given a (real or possibly vector valued) function $f\in L^1(\T^d)$ or a finite (possibly signed) measure $\mu$ on $\T^d$, for $\xi \in \Z^d$, we write
\[ \hat f (\xi) = \int_{\T^d} \exp\bra{ - 2 \pi i \xi \cdot x} f(x) \,\d x, \quad \hat \mu (\xi) = \int_{\T^d} \exp\bra{ - 2 \pi i \xi \cdot x}  \,\d \mu (x),\]
for their Fourier transforms. We will use throughout Plancherel identity, for $f \in L^2(\T^d)$,
\begin{equation}\label{eq:plancherelidentity}
\int_{\T^d} |f|^2(x) \,\d x = \sum_{\xi \in \Z^d} | \hat f |^2(\xi),
\end{equation}
as well as the following generalization to higher (even) exponents $p = 2 m$, $m \ge 1$: for (possibly vector valued) $f \in L^p(\T^d)$,
\begin{equation}\label{eq:plancherel-higher} \int_{\T^d} | f |^{p}(x) \,\d x  =  \sum_{ \xi \in (\Z^d)^p} \prod_{i=1}^p  \hat{f}(\xi_i) \delta_0\bra{ \xi_1 +  \ldots + \xi_p},\end{equation}
that can be seen formally by using the inversion formula
\[ f(x) = \sum_{\xi \in \Z^d} \exp\bra{ 2 \pi i \xi\cdot x} \hat{f}(\xi),\]
developing the $p$-th  power $|f(x)|^p$ and using the  orthogonality 
\[ \int_{\T^d} \exp\bra{ 2 \pi i \xi \cdot x} \,\d x = \delta_0(\xi).\]
Notice that a suitable application of Young convolution inequality on $\Z^d$ to the right-hand side of \eqref{eq:plancherel-higher} yields a special case of the classical Hausdorff-Young inequality
\begin{equation}\label{eq:hausdorff-young}  \bra{ \int_{\T^d} | f |^{p}(x) \,\d x}^{1/p}  \le   \bra{ \sum_{\xi \in \Z^d} |\hat{f}|^{q}(\xi) }^{1/q},\end{equation}
with $q=p/(p-1)$ conjugate exponent, which in fact holds for every $p \ge 2$, not necessarily even integer.
The heat semigroup $(P_t)_{t \ge 0}$ on $\T^d$ can be defined in various equivalent ways, via convolution with respect to a Gaussian kernel
\[ P_t \mu  = \mu * G_t, \quad \text{with $G_t(x) =  \sum_{k \in \Z^d} \exp\bra{ -|x+k|^2/(2t)} (2 \pi t)^{-d/2}$}\]
or  via the Fourier multiplier 
\begin{equation}\label{eq:fourier-heat} \widehat{ P_t \mu}(\xi) =  \exp\bra{- 2 \pi^2 t |\xi|^2} \hat{\mu}(\xi), \quad \text{for $\xi \in \Z^d$}.\end{equation}
We will crucially use the fact that the semigroup transforms possibly singular measures $\mu$ into smooth densities. The generator of $(P_t)_{t \ge 0}$ is (half) the Laplacian operator $\frac{1}{2} \Delta$. We will use also the Fourier transforms of the solution to the Poisson equation $-\Delta f = \mu$, with  $\hat{\mu}(0) = \mu(\R^d)  =0$,
\[  \hat{f}(\xi) = \frac{1}{4 \pi^2} \frac{\hat{\mu}(\xi)}{|\xi|^2},\]
and its gradient
\begin{equation}\label{eq:gradient-fourier} 
\widehat{ \nabla f }(\xi) =\frac{i \xi }{2\pi  } \frac{\widehat{\mu}\bra{\xi}}{|\xi|^2}.
\end{equation}

\subsection{Wasserstein distance}

Let $p \ge 1$. On a metric space $(X, \dist)$, we say that a Borel measure $\mu$ has  finite $p$-th moment if
\[ \int_X \dist(x, x_0)^p \,\d \mu(x)< \infty,\]
for some (hence any) $x_0 \in X$. Given two measures $\mu$, $\nu$ with same total mass $\mu(X) = \nu(X)$, and finite $p$-th moment, their $p$-th Wasserstein distance is given by
\begin{equation}\label{es:wasserstein} W_p(\mu, \nu) := \inf\brag{ \int_{X\times X} \dist(x,y)^p \,\d \pi(x,y) \, : \, \pi \in \Gamma(\mu, \nu)}^{1/p},\end{equation}
where $\Gamma(\mu, \nu)$ is the set of couplings between $\mu$, $\nu$, i.e., measures on $X\times X$ whose first and second marginals are respectively $\mu$ and $\nu$. A simple application of H\"older's inequality gives that, for $1 \le p \le q$,
\begin{equation}\label{eq:monotonicity} W_p^p(\mu, \nu) \le \mu(X)^{1-p/q} \bra{ W_q^q(\mu, \nu)}^{p/q}.\end{equation}

For every $p\ge 1$, the definition \eqref{es:wasserstein} above yields a distance, in particular the triangle inequality holds. The Kantorovich dual formulation reads, for $p =1$,
\begin{equation}\label{eq:kantorovich} W_p(\mu, \nu) = \sup \brag{ \int_X f \,\d (\mu -\nu) \, : \, \Lip(f) \le 1},\end{equation}
where the Lipschitz constant $\Lip(\cdot)$ is given by
\[ \Lip \bra{f} := \sup_{ x\neq y} \frac{ |f(y)-f(x)|}{\dist(x,y)}.\]
For $p>1$, duality is slightly more involved but  we will not need it.

\medskip
In our setting $X = \T^d$,  many tools are at our disposal to bound the Wasserstein distance in terms of (possibly) simpler quantities to analyse. In particular, we will make use of the following result (see \cite{bobkov2021simple} for related bounds).

\begin{lemma}\label{lem:general-transport-pde-bound}
Let $\mu$, $\nu$ be measures both on $\T^d$ with same total mass $\mu(\T^d) = \nu(\T^d)$ and for every $\varepsilon>0$ let $u_\varepsilon \in H^1(\T^d)$ solve the Poisson equation
$$ -\Delta u_\varepsilon = P_\varepsilon (\mu - \nu).$$
Then, there exists $C = C(d)>0$ such that
\begin{equation}\label{eq:upper-p} W_1( \mu, \nu) \le  \inf_{\epsilon>0} \brag{ C \mu(\mathbb{T}^d) \epsilon^{1/2}  + \|\nabla u_\varepsilon \|_{L^2}},\end{equation}
and
\begin{equation}\label{eq:lower-p} W_1(\mu,  \nu)  \ge \sup_{M, \epsilon >0} \brag{  \frac 1 M \| \nabla u_\varepsilon \|^2_{L^2}  - \frac C {M^3} \| \nabla u_\varepsilon\|_{L^4}^4 }.\end{equation}
If moreover $\nu = \mu(\T^d) \mathscr{L}^d_{\T^d}$ has constant density, then, for $p > 1$, there exists $C = C(d,p)>0$ such that
\begin{equation}\label{eq:wasserstein-p>1} W_p^p(\mu, \mu(\T^d) \mathscr{L}^d_{\T^d} ) \le C \inf_{\epsilon>0} \brag{  \mu(\mathbb{T}^d) \epsilon^{p/2}  + \mu(\mathbb{T}^d)^{1-p} \norm{ \nabla u_\varepsilon }_{L^p}^p}.\end{equation}
\end{lemma} 

\begin{proof}
By a simple rescaling, we can assume that $\mu$ and $\nu$ are probability measures. The triangle inequality gives
\begin{equation}\label{eq:triangle-ineq-W1-upper}
W_1 \bra{\mu, \nu} \leq W_1 \bra{\mu, P_\epsilon\mu} + W_1 \bra{\nu, P_\epsilon\nu} + W_1 \bra{P_\epsilon \mu, P_\epsilon\nu}.
\end{equation}
Note that for any $p\geq 1, \epsilon > 0$ and for any probability $\mu$, one has
\begin{equation}\label{eq:upper-bound-smoothed-Wp}
W_p \bra{\mu, P_\epsilon \mu} \leq C {\epsilon}^{1/2}.
\end{equation}
Indeed, since $\mu \otimes P_\epsilon \mu$ is an admissible candidate for the transport problem, we may write
\[
\begin{split}
W_p^p \bra{\mu, P_\epsilon \mu} & \leq  \int_{\T^d \times \T^d } \dist(x,y)^p \,\d \mu(x) \,\d P_\epsilon \mu (y) \\ 
& \leq \int_{\T^d \times \T^d} \sum_{k \in \mathbb{Z}^d} \abs{x-y-k}^p \exp \bra{ -\abs{x-y-k}^2/(2 \varepsilon)}\bra{2 \pi \varepsilon}^{- d/2} \,\d y \,\d \mu (x) \\
& \le C  \varepsilon^{-d/2} \int_{\R^d} \abs{u}^p e^{-\abs{u}^2} \,\d u \leq C \epsilon^{\frac p 2}. 
\end{split}
\]
The latter and \eqref{eq:triangle-ineq-W1-upper} combine to
\[
W_1 \bra{\mu, \nu} \le C  {\epsilon}^{1/2} + W_1 \bra{P_\epsilon \mu, P_\epsilon\nu}.
\]
By \eqref{eq:kantorovich} and integrating by parts we may write
\[
\begin{split}
W_1 \bra{P_\epsilon \mu, P_\epsilon\nu} & = \sup_{\Lip(f) \le 1} -\int_{\mathbb{T}^d} f \Delta u_\varepsilon  = \sup_{\Lip(f) \le 1} \int_{\mathbb{T}^d} \nabla f \nabla u_\varepsilon\\
& \le \int_{\mathbb{T}^d} | \nabla u_\varepsilon| \le \| \nabla u_\varepsilon \|_{L^2},
\end{split}
\]
where in the last line we used the Cauchy-Schwarz inequality. 

\medskip
Let us turn to \eqref{eq:lower-p}. By \cite[Lemma~5.2]{wang2019limit} for any constant $M > 0$ there exists an $M$-Lipschitz function $f^M$ such that 
\begin{equation}\label{eq:est-neq-set-LusLip}
\abs{ \brag{f^M \neq u_\varepsilon } } \leq \frac C {M^4} \| \nabla u_\varepsilon \|_{L^4}^4.
\end{equation}
By the contraction property of the heat kernel (see for instance \cite[Lemma~5.2]{santambrogio2015optimal}) and \eqref{eq:kantorovich}, arguing as before, we may write
\[
\begin{split}
W_1 \bra{\mu , \nu} & \ge W_1 \bra{P_\epsilon \mu, P_\epsilon \nu} \\
& \ge \frac1M \int_{\mathbb{T}^d} \nabla f^M \nabla u_\varepsilon \\
& = \frac1M \| \nabla u_\varepsilon \|_{L^2}^2 + \frac1M \int_{\mathbb{T}^d} \bra{\nabla f^M - \nabla u_\varepsilon } \nabla u_\varepsilon.
\end{split}
\]
In view of \eqref{eq:lower-p} it remains to show that
\begin{equation}\label{eq:stimaintegrale}
\frac1M \int_{\mathbb{T}^d} \bra{\nabla f^M - \nabla u_\varepsilon } \nabla u_\varepsilon \ge - \frac C{M^3} \|\nabla u_\epsilon\|_{L^4}^4.
\end{equation}
Note that the integral on the right hand side vanishes on the set $\{ f^M = u_\varepsilon\}$ due to the locality of the gradient. Furthermore, by Young's product inequality and H\"older's inequality
\[
\begin{split}
   \abs{\frac1M \int_{\{f^M \neq u_\varepsilon\}} \bra{\nabla f^M - \nabla u_\varepsilon } \nabla u_\varepsilon} & \leq M \abs{ \brag{f^M \neq u_\varepsilon }} + \frac3 {2M} \int_{\{f^M \neq u_\epsilon\}} \abs{\nabla u_\varepsilon }^2 \\
 &  \leq M \abs{ \{f^M \neq u_\varepsilon \}}  \\
 & \quad \quad + \frac 3{2M} \abs{ \brag{f^M \neq u_\varepsilon }}^\frac12   \bra{\int_{\{f^M \neq u_\varepsilon\} } \abs{\nabla u_\varepsilon }^4}^\frac12 \\
&  \stackrel{\eqref{eq:est-neq-set-LusLip}}{\leq} \frac C {M^3} \| \nabla u_{\varepsilon}\|_{L^4}^4, 
\end{split}
\]
which yields \eqref{eq:stimaintegrale}.

\medskip
Finally, let us turn to \eqref{eq:wasserstein-p>1}, by the triangle inequality, we estimate
\[
W_p \bra{\mu, 1} \leq W_p \bra{\mu, P_\epsilon \mu} + W_p \bra{P_\epsilon \mu , 1} \stackrel{\eqref{eq:upper-bound-smoothed-Wp}}{\le} C \epsilon^{1/2} + W_p \bra{P_\epsilon \mu, 1}.
\]
To estimate the second term on the right hand side, by uniformity of the second measure we may apply \cite[Theorem 2]{ledoux2017optimal} to get (recall that $\nu = \mathscr{L}^d_{\mathbb{T}^d}$)
\[
W_p \bra{P_\epsilon \mu, 1} \leq p \| \nabla  u_\epsilon \|_{L^p},
\]
which yields \eqref{eq:wasserstein-p>1} after taking $p$-th powers and using the inequality $(x+y)^p \le 2^{p-1} (x^p + y^p)$ for non-negative $x$, $y$.
%
%
\end{proof}

\subsection{Fractional Brownian Motion}

In this section we recall the definition and basic properties of fractional Brownian motion, referring to \cite{nourdin2012selected} for a detailed study of such processes and their properties.


\begin{definition}\label{definizionefBm1d}
Let $H \in (0,1)$. A (real) fractional Brownian motion with Hurst index $H$ is a Gaussian stochastic process $B^H = \bra{B^H_t}_{t \ge 0}$ with continuous paths, centered and with covariance function
\begin{equation}\label{covarianzafbm}
\Ex\left [B^H_tB^H_s \right ] = \frac{1}{2} \left ( s^{2H} + t^{2H} - \left | t - s \right |^{2H} \right ).
\end{equation}
A fractional Brownian motion with values in $\R^d$ is given by $d$ independent (real) fractional Brownian motions, all with the same Hurst index $H$. A fractional Brownian motion with values in $\T^d$ is the natural projection of a fractional Brownian motion with values in $\R^d$.
\end{definition}

\begin{remark}
As is well-known, the case $H=1/2$ reduces to a standard Brownian motion. 
Notice that our definition yields $B_0^H = 0$, but different starting points $x \in \R^d$ may be considered, introducing e.g.\ the process $B_t^H +X$ with an independent random variable $X \in \R^d$. For simplicity, we limit ourselves to the case $B_0^H = 0$ in our results.
\end{remark}

We now define the occupation and empirical measures of a fractional Brownian motion.

\begin{definition}[occupation and empirical measure]
Let $B^H$ be a fractional Brownian motion taking values in $\T^d$. For every $T\ge 0$, its occupation measure is the random measure 
\[
\mu_T := \int_0^T \delta_{B_t^H} \,\d t, 
\]
i.e.,  for every bounded Borel function $f$ on $\T^d$,
\[ \int_{\T^d} f  \,\d \mu_T = \int_0^T f(B_t^H) \,\d t.\]
We define the empirical measure of a fractional Brownian motion $B^H$ as the occupation measure renormalized to a probability measure, i.e. the quantity $\mu_T/T.$
\end{definition} 

We notice that $\mu_T(\T^d) = T$. 
We also consider a discrete approximation of $\mu_T$ given by
\begin{equation}\label{eq:discrete-empirical}
\mu_{\tau,T} := \sum_{t=1}^{\lfloor T/\tau \rfloor} \delta_{B_{t \tau}^H} \tau,
\end{equation}
where $\tau>0$ and with total mass is $\mu_{\tau, T}(\T^d) = \lfloor T/\tau \rfloor \tau$.

\medskip
In the proof of our results we need precise bounds on the (mixed) $p$-th moments of the Fourier transform of the occupation measure of $B^H$. The key property that we use to simplify our estimates is the \emph{local  non-determinism} of $B^H$, i.e., the lower bound on the covariance operator, for some $C = C(H,p)>0$
\begin{equation}\label{eq:local-non-determinism}
\operatorname{Cov}\bra{ B^H_{t_1} - B^H_{t_0}, B^H_{t_2} - B^H_{t_1}, \ldots, B^H_{t_p} - B^H_{t_{p-1}}} \ge C \operatorname{diag}( |t_1- t_0|^{2H}, \ldots, |t_p - t_{p-1}|^{2H}),
\end{equation}
for any choice $0 \le t_0 < t_1 < \ldots < t_p \leq T$, where the inequality is in the sense of quadratic forms. A full proof of this fact, valid for $H \in (0,1)$ is given in \cite[Section~2.1]{xiao2006properties}, see also \cite[Section~2.4]{galeati2020prevalence} for a simpler argument.


%
%
%
%


\section{Proof of Theorem~\ref{thm:main-Td}} \label{sec:thm:main}

In this section, we let $d\ge 1$, $H \in (0,1)$ and $B = B^H$ be a fractional Brownian motion with Hurst index $H$, with values in $\T^d$, and write $\mu_T = \int_0^T \delta_{B_s} \,\d s$, for $T \ge 0$, for its occupation measure.  Given $\epsilon>0$, we consider a solution $u_{\varepsilon}$ to the Poisson PDE
\begin{equation}\label{eq:defueps}
-\Delta u_{\varepsilon} = P_{\epsilon} (\mu_T - T).
\end{equation}

For clarity of exposition, we split the proof into four parts. First, we collect some useful upper and lower bounds on moments of the Fourier transform of the empirical measure. Next, we prove the asymptotic upper bound in the case  $p=1$ (Proposition~\ref{prop:upper-bound-p-1-torus}), essentially because it is a simple argument and it suggests us how to  optimally choose $\epsilon = \epsilon(d,H,T)$ in Lemma~\ref{lem:general-transport-pde-bound}. Then, we prove a general bound for the expectations of the Lebesgue norms of $\nabla u_\epsilon$ (Proposition~\ref{prop:p-moments-torus}). Finally, we deduce at once the asymptotic lower bound for $p=1$ (and in fact for every $p\ge1$) as well as the upper bounds for every $p\ge 1$ (Proposition~\ref{prop:lower-bound}).

\subsection{Fourier transform moment bounds}

\begin{lemma}\label{lem:upper-bound-p}
For every $p \in \mathbb{N}$, 
for every $T \ge 0$ and $\xi \in (\Z^d)^p$,
\begin{equation}\label{eq:upper-bound-fourier-modes}
\abs{ \Ex\braq{ \prod_{j=1}^p\widehat{\mu_T}(\xi_j) }} \lesssim  \sum_{\sigma \in \mathcal{S}_p} \prod_{j=1}^{p} \min\brag{ \frac{1}{| \sum_{i=1}^j \xi_{\sigma(i)} |^{1/H}}, T}.
\end{equation} 
\end{lemma}


\begin{proof}
By definition,
\[ \widehat{\mu_T}(\xi_j) = \int_0^T \exp\bra{ 2\pi i \xi B_{t_j}} \,\d t_j,\]
so that
\[ \prod_{i=1}^p \widehat{\mu_T}(\xi_i) = \int_{[0,T]^p} \exp\bra{ 2 \pi i \sum_{j=1}^p \xi_j B_{t_j}} \,\d t_1 \ldots \,\d t_p.\]
We split integration over $[0,T]^p$ into $p!$ simplexes, one of every $\sigma \in \mathcal{S}_p$,
\[ \Delta_\sigma := \brag{ 0 \le t_{\sigma(1)} \le \ldots  \le t_{\sigma(p)} \le T}.\]
We now argue only in the case $\sigma$ being the identity permutation, the other cases being analogous. A summation by parts gives 
\[ \sum_{j=1}^p \xi_j B_{t_j} =   B_{0} \sum_{i=1}^p \xi_i + \sum_{j=1}^{p}  (B_{t_j}- B_{t_{j-1}}) \sum_{i=j}^{p} \xi_i = \sum_{j=1}^{p}  (B_{t_j}- B_{t_{j-1}}) \sum_{i=j}^{p} \xi_i,\]
where we let $t_0 = 0$ and we use that $B_0= 0$. 

Using this identity, the Fourier transform (characteristic function) of a Gaussian random variable and \eqref{eq:local-non-determinism}, it follows that
\[ \Ex\braq{ \exp\bra{ 2 \pi i \sum_{j=1}^p \xi_j B_{t_j}}} \le \exp\bra{- \frac C 2 \sum_{j=1}^{p} \abs{ \sum_{i=j}^{p} \xi_i  }^2 |t_{j}- t_{j-1}|^{2H}}.\]
We then bound from above the integral
\[ \begin{split}  \int_{\Delta_\sigma} \Ex\braq{ \exp\bra{ 2 \pi i \sum_{j=1}^p \xi_j B_{t_j}}} &  \,\d t_1 \ldots \,\d t_p \\
 & \le  \int_{\Delta_\sigma}  \exp\bra{- \frac C 2 \sum_{j=1}^{p} \abs{ \sum_{i=j}^{p} \xi_i  }^2 |t_{j}- t_{j-1}|^{2H}} \,\d t_1 \ldots \,\d t_p\\
& \le \int_{[0,T]^p }\exp\bra{- \frac C 2 \sum_{j=1}^{p} \abs{ \sum_{i=j}^{p} \xi_i  }^2 s_j^{2H}} \,\d s_1 \ldots \,\d s_p,
\end{split}\]
where we performed the change of variables $s_j= t_j- t_{j-1}$, for $j \ge 1$, recalling that $t_0=0$. To conclude, we split into a product of $p$ integrals that we bound separately
\[ \int_0^T \exp\bra{- \frac C 2  \abs{ \sum_{i=j}^{p} \xi_i  }^2 s_j^{2H}}  \,\d s_j \le \min\brag{ \frac{C}{ \abs{ \sum_{i=j}^{p} \xi_i}^{1/H}}, T},\]
for a (possibly different) constant $C = C(p,d, H)>0$.
\end{proof}


Letting $p=2$ and $\xi: = \xi_1  = -\xi_2 \neq 0$ in \eqref{eq:upper-bound-fourier-modes} yields the upper bound
$$  \Ex\braq{  \abs{ \widehat{\mu_T}(\xi)}^2  } \le C \frac{T}{|\xi|^{1/H}}. $$
For our purposes, we need also a companion lower bound, that can be simply obtained as the next lemma shows.

\begin{lemma}\label{lem:second-moment-continuous}
For $T$ sufficiently large (depending on $H$ and $d$ only) it holds, for every $\xi \in \Z^d$, $\xi\neq 0$,
\begin{equation}\label{eq:lower-bound-fourier-modes}
\Ex\braq{ \abs{ \widehat{\mu_T}(\xi)}^2 } \sim T / |\xi|^{1/H}.
 \end{equation}
\end{lemma}

\begin{proof}
Thanks to Lemma \ref{lem:upper-bound-p} we only need to prove the lower bound in \eqref{eq:lower-bound-fourier-modes}. By definition of the Fourier transform we may write
\[ \begin{split}
 \Ex\braq{ \abs{  \widehat{\mu_T}(\xi) }^2 } & = \int_0^T \int_0^T \Ex\braq{ \exp\bra{ 2 \pi i \xi \cdot (B_{t} - B_s)}} \,\d s \,\d t  \\
 & = \int_0^T\int_0^T  \exp\bra{ -  2 \pi^2| \xi|^2|t-s|^{2H}} \,\d s \,\d t \\
 & \ge   \int_0^{T/2} \int_0^{T -s} \exp\bra{-2\pi^2|\xi|^2 t^{2H}} \,\d t \,\d s  \\ 
 &\ge \int_0^{T/2} \int_0^{T/2} \exp\bra{-2\pi^2|\xi|^2 t^{2H}} \d t \,\d s\\
 &  \ge C \frac{T}{ |\xi|^{1/H}} \int_0^{T |\xi|^{1/H}/2} \exp\bra{-2\pi^2t^{2H}} \,\d t \\
 & \ge C\frac{T}{ |\xi|^{1/H}} \int_0^{1/2} \exp\bra{-2\pi^2 t^{2H}} \,\d t,
 \end{split}\]
 where the last inequality holds if $T \ge 1$, so that $T |\xi|^{1/H}/2 > 1/2$. 
\end{proof}

\subsection{Upper bound, case $p=1$}

\begin{lemma}\label{lem:g}
Let $d \ge 1$, $\epsilon>0$ and define, for $\xi \in \Z^d$, 
 \[ g(\xi) = \frac{ \exp\bra{ - \epsilon |\xi|^2}}{|\xi|+1}.\]
 Then, for every  $p \ge 1$,
 \[ \| g \|_{\ell^p(\Z^d)} \sim  \begin{cases} 1 & \text{if $d<p$,}\\
 |\log \epsilon|^{1/p} & \text{if $d=p$,}\\
 \epsilon^{-\frac 1 2 \bra{ d/p -1 }} & \text{ if $d>p$.}\end{cases}\]
 \end{lemma}
 
\begin{proof}
By comparing the series with the integral and using polar coordinates we may write
\[ \sum_{\xi \in \Z^d} g(\xi)^p \sim \int_{|\xi| \ge 1} \frac{\exp\bra{-p \epsilon|\xi|^2}}{|\xi|^{p}} \,\d \xi =  \int_1^\infty\frac{\exp\bra{-p \epsilon r^2}}{r^{p}}  r^{d-1} \,\d r.\]
Note that if $d<p$ we may estimate
\[
\int_1^\infty \exp\bra{-p \epsilon r^2} r^{d-1 -p} \,\d r \sim \int_1^\infty r^{d-1 -p} \,\d r \sim 1.
\]
Otherwise, we may split the integral and write
\[
\begin{split}
\int_1^\infty \exp\bra{-p \epsilon r^2} r^{d-1 -p} \,\d r & = \int_1^{\frac{1}{\sqrt{\epsilon}}} \exp\bra{-p \epsilon r^2} r^{d-1 -p} \,\d r + \int_{\frac{1}{\sqrt{\epsilon}}}^\infty \exp\bra{-p \epsilon r^2} r^{d-1 -p} \,\d r \\
& = I_1 + I_2.
\end{split}
\]
We first estimate the former term on the right hand side
\[
I_1 \sim \int_1^{\frac{1}{\sqrt{\epsilon}}} r^{d-1-p} \,\d r \sim \begin{cases}
\abs{\log \epsilon} & \text{if}\ d=p, \\
\epsilon^{-\frac12 \bra{d-p}} & \text{if}\ d>p.
\end{cases}
\]
Finally, by the change of variable $s = \sqrt{p \epsilon} r$ we can bound the latter term on the right hand side by
\[
I_2 \sim \epsilon^{-\frac12 \bra{d-1-p}} \int_1^\infty e^{- s^2} s^{d-1-p} \,\d s \sim \epsilon^{-\frac12 \bra{d-1-p}},
\]
which concludes the proof.
\end{proof}

\begin{proposition}\label{prop:upper-bound-p-1-torus}
Define $\epsilon = \epsilon(d,H,T)>0$ as follows:
\begin{equation}\label{eq:choice-epsilon} \sqrt{\epsilon} = \begin{cases} T^{-1/2}& \text{if $d< 2+\frac1H$,}\\
								  ( \log (T) / T)^{1/2} & \text{if $d = 2+\frac1H$,}\\
								  T^{-\frac{1}{d-1/H}} & \text{if $d>2+\frac1H$,} \end{cases}
\end{equation}
and let $u_\epsilon$ be a solution to \eqref{eq:defueps}. Then, as $ T \to \infty$,
\begin{equation}\label{eq:asymptotics-L2-gradient}
\Ex\braq{\norm{ \nabla u_\epsilon }_{L^2}} \sim  T \sqrt{\epsilon},
\end{equation}
hence, by \eqref{eq:upper-p},
\[ \Ex\braq{W_1(\mu_T, T) } \lesssim T \sqrt{\epsilon}. \]
\end{proposition}

\begin{proof}
By Plancherel's identity (see \eqref{eq:plancherelidentity}) and Lemma~\ref{lem:upper-bound-p} we may estimate the second moment of $\nabla u_\epsilon$ by
\[ \begin{split} \Ex\braq{\norm{ \nabla u_{\epsilon} }_{L^2}^2} & = (2 \pi)^{-2} \sum_{\xi \in \Z^d \setminus\brag{0}} \Ex\braq{\abs{\widehat{\mu_T}(\xi)}^2} \frac{\exp\bra{-\epsilon|\xi|^2/2 }}{|\xi|^2}\\
& \sim   T  \sum_{\xi \in \Z^d \setminus\brag{0}} \frac{\exp\bra{-\epsilon|\xi|^2/2}}{|\xi|^{2+1/H}}\\
& \sim T \norm{ g }_{\ell^{2 +1/H}(\Z^d)}^{2+1/H}
\end{split}\]
where $C = C(d,H)>0$ and $g$ is as in Lemma~\ref{lem:g} with $\epsilon/(2+1/H)$ instead of $\epsilon$. By Lemma \ref{lem:g} we have
 \[ \norm{ g }_{\ell^{2 +1/H}(\Z^d)}^{2+1/H}  \sim \begin{cases} 1 & \text{if $d<2+\frac1H$,}\\
 |\log \epsilon|  & \text{if $d=2+\frac1H$,}\\
 \epsilon^{-\frac 1 2 \bra{ d -2 - 1/H }} & \text{ if $d>2+\frac1H$.}
 \end{cases}\]
Finally taking the square root, we obtain \eqref{eq:asymptotics-L2-gradient},  indeed
\[ \sqrt{T  \norm{g}_{\ell^{2 +1/H}(\Z^d)}^{2+1/H}} \sim \begin{cases} \sqrt{T} = T \cdot T^{-1/2} & \text{if $d<2+\frac1H$,}\\
 \sqrt{T |\log \epsilon|} \sim T \cdot \bra{ (\log T)/T  }^{1/2}  & \text{if $d=2+\frac1H$,}\\
 \sqrt{T} \epsilon^{-\frac 1 4 \bra{ d -2 - 1/H }}
 = T \cdot T^{-\frac{1}{d-1/H}} & \text{if $d>2+\frac1H$.}
 \end{cases}\]
\end{proof}

\subsection{Moment bounds for $\nabla u_\epsilon$}

Before we move to the second part of the argument, we need a generalized Young convolution inequality, which can be easily proved by induction.

\begin{lemma}\label{lem:young}
Let $p \in \mathbb{N}$, $p\ge 2$, let $f_1, \ldots, f_p: \Z^d \to [0, \infty]$, $F_2, \ldots, F_p: \Z^d \to [0, \infty]$ be measurable and define
\[ G_p (\xi_1, \ldots, \xi_p) := \prod_{i=1}^p f_i(\xi_i) \prod_{j=2}^p F_j\bra{ \sum_{i=1}^j \xi_i}.\]
Let  $\lambda_i, \Lambda_j \in [1, \infty]$, for $i=\brag{1, \ldots, p}$, $j \in \brag{2, \ldots, p}$ such that
\begin{equation}\label{eq:lemyoungcond}
\begin{split}
&\frac1{\lambda_1} + \frac1{\lambda_2} + \frac1{\Lambda_2} = 2, \\
&\frac1{\lambda_k} + \frac1{\Lambda_k} = 1 \quad \mbox{if}\; 3\le k \le p.
\end{split}
\end{equation}
Then,
\[ \norm{G_p}_{\ell^1(\Z^{d\times p})} \le \prod_{i=1}^p \norm{f_i}_{\ell^{\lambda_i}(\Z^d)}\prod_{j=2}^p \norm{F_j}_{\ell^{\Lambda_j}(\Z^d)}.\]
\end{lemma}
\begin{proof}
Our argument relies on an iterative application of H\"older's inequality and Young's convolution inequality. We argue by induction over $p\ge 2$. Let us consider the case $p=2$.  Let us choose $\alpha$ such that
\[  \frac{1}{\alpha} + \frac{1}{\Lambda_2}= 1,\]
so that by H\"older's inequality we may write
\begin{equation}\label{eq:lemyoungbase}
\begin{split}
\norm{ G_2 }_{\ell^1 (\Z^{d \times 2})}& = \norm{ (f_1 * f_2) F_2 }_{\ell^1(\Z^{d \times 2})} \le \norm{ f_1 * f_2}_{\ell^{\alpha}(\Z^{d})} \norm{F_2}_{\ell^{\Lambda_2}(\Z^{d})}.
\end{split}
\end{equation}
Note that by the assumption $\Lambda_2 \in [1, \infty]$ it follows that $\alpha \ge 1$, thus we can choose $\alpha$ being such that
\[ \frac{1}{\alpha} = \frac{1}{\lambda_1}+\frac{1}{\lambda_2}-1,\]
which ensures \eqref{eq:lemyoungcond}. By Young's convolution inequality we may estimate the first term on the right hand side of \eqref{eq:lemyoungbase}
\[\norm{ f_1 * f_2}_{\ell^{\alpha}(\Z^{d})} 
 \le \norm{f_1}_{\ell^{\lambda_1}(\Z^d)}\norm{f_2}_{\ell^{\lambda_2}(\Z^d)}.  \]
The latter and \eqref{eq:lemyoungbase} combines to
\[\begin{split} \norm{ G_2 }_{\ell^1 (\Z^{d \times 2})}& \le \norm{f_1}_{\ell^{\lambda_1}(\Z^d)}\norm{f_2}_{\ell^{\lambda_2}(\Z^d)}  \norm{F_2}_{\ell^{\Lambda_2}(\Z^d)}.
\end{split}\]
%
Assuming that the thesis holds for $p-1 \ge 2$, we argue similarly to obtain it for $p$. Note that by definition $G_p =  (G_{p-1} * f_p) F_p$. Choosing $\lambda_p, \Lambda_p$ as in \eqref{eq:lemyoungcond} we may apply H\"older's inequality and Young's convolution inequalities to get
\[\begin{split} \norm{ G_p }_{\ell^{1}(\Z^{d \times p})}& = \norm{ (G_{p-1} * f_p) F_p }_{\ell^{1}(Z^{d \times p})}  \\
& \le \norm{G_{p-1}}_{\ell^{1}(\Z^{d\times (p-1)})}\norm{f_2}_{\ell^{\lambda_p}(\Z^{d})}  \norm{F_p}_{\ell^{\Lambda_p}(\Z^d)}.
\end{split}\]
Finally, by the inductive assumption the thesis holds. 
\end{proof}

We are now in a position to prove the key upper bound for even integral moments of $\nabla u_\epsilon$.

\begin{proposition}\label{prop:p-moments-torus}
Let $\epsilon = \epsilon(d,H,T)>0$ be as in \eqref{eq:choice-epsilon} 
and let $u_\epsilon$ be a solution to \eqref{eq:defueps}. Then, for every even $p \in \mathbb{N}$,  for every $T\ge 0$ sufficiently large,
\begin{equation}\label{eq:p-moment-gradient}
\Ex\braq{ \norm{ \nabla u_\epsilon}_{L^p}^p} \lesssim  (T \sqrt{\epsilon})^p.
\end{equation}
\end{proposition}
\begin{proof}
{\sc Step 1.} We argue that in order to control the $p$-th moment of $\nabla u_\epsilon$ it is enough to show that there exists a constant $C$ such that
\begin{equation}\label{eq:general-sum-no-permutation} 
T \sum_{\xi_1, \ldots, \xi_p \in \Z^d\setminus \brag{0}}   \prod_{i=1}^p\frac{\exp\bra{- \epsilon |\xi_{i}|^2/2} }{|\xi_i|}  \prod_{j=1}^{p-1} \min\brag{ \frac{1}{| \sum_{i=1}^j \xi_{i} |^{1/H}}, T} \delta_0\bra{\sum_{i=1}^p \xi_{i}} \le C(T\sqrt{\epsilon})^p.
\end{equation}
Indeed, by \eqref{eq:fourier-heat}, \eqref{eq:gradient-fourier} and \eqref{eq:plancherel-higher} it follows that
\[\begin{split} \Ex&\braq{ \norm{ \nabla u_\epsilon }_{L^p}^p} \\
& \quad \quad \le C \sum_{\xi_1, \ldots, \xi_p \in \Z^d\setminus \brag{0}}  \prod_{i=1}^p\frac{\exp\bra{- \epsilon |\xi_i|^2/2} }{|\xi_i|} \abs{ \Ex\braq{ \prod_{i=1}^p \widehat{\mu_T}(\xi_i)}} \delta_0\bra{\sum_{i=1}^p \xi _i}.
\end{split}\]
By Lemma~\ref{lem:upper-bound-p}, we further bound from above the right hand side by a constant times
\[   T \sum_{\xi_1, \ldots, \xi_p \in \Z^d\setminus \brag{0}}   \prod_{i=1}^p\frac{\exp\bra{- \epsilon |\xi_i|^2/2} }{|\xi_i|} \sum_{\sigma \in \mathcal{S}_p} \prod_{j=1}^{p-1} \min\brag{ \frac{1}{| \sum_{i=1}^j \xi_{\sigma(i)} |^{1/H}}, T} \delta_0\bra{\sum_{i=1}^p \xi_i},\]
where the first $T$ term is due to the condition $\sum_{i=1}^p \xi_i = 0$.
Since, for every $\sigma \in \mathcal{S}_p$,
\[ \prod_{i=1}^p\frac{\exp\bra{- \epsilon |\xi_i|^2/2} }{|\xi_i|} =  \prod_{i=1}^p\frac{\exp\bra{- \epsilon |\xi_{\sigma(i)}|^2/2} }{|\xi_{\sigma(i)}|} \quad \text{and} \quad \delta_0\bra{\sum_{i=1}^p \xi_i} = \delta_0 \bra{\sum_{i=1}^p \xi_{\sigma(i)}},\]
we can exchange summations and reduce the problem to bound from above, for every $\sigma \in \mathcal{S}_p$,
\[ T \sum_{\xi_1, \ldots, \xi_p \in \Z^d\setminus \brag{0}}   \prod_{i=1}^p\frac{\exp\bra{- \epsilon |\xi_{\sigma(i)}|^2/2} }{|\xi_{\sigma(i)}|}  \prod_{j=1}^{p-1} \min\brag{ \frac{1}{| \sum_{i=1}^j \xi_{\sigma(i)} |^{1/H}}, T} \delta_0\bra{\sum_{i=1}^p \xi_{\sigma(i)}},\]
but the quantity above is independent of $\sigma$, after changing summation variables. Therefore, we may assume that $\sigma$ is the identity permutation, which establishes \eqref{eq:general-sum-no-permutation}.

\medskip
{\sc Step 2.} We argue by an induction argument that in order to prove \eqref{eq:general-sum-no-permutation} it is enough to show that for every $p$ there exists a constant $C$ such that
\begin{equation}\label{eq:key-step} T \sum_{\xi_1, \ldots, \xi_p}^\star   \prod_{i=1}^p\frac{\exp\bra{- \epsilon |\xi_{i}|^2/2} }{|\xi_{i}|}  \prod_{j=1}^{p-1}  \frac{1}{| \sum_{i=1}^j \xi_{i} |^{1/H}} \delta_0\bra{\sum_{i=1}^p \xi_{i}} \le C (T \sqrt{\epsilon})^p,
\end{equation}
where the symbol $\sum^\star$ denotes the summation restricted upon
\[ \text{$\xi_1, \ldots, \xi_p \in \Z^d \setminus \brag{0}$ such that $\sum_{i=1}^j \xi_i \neq 0$ for every $j=1, \ldots, p-1$.}\]
Indeed, once \eqref{eq:key-step} is established for every $p \in \mathbb{N}$, what remains to bound in the summation \eqref{eq:general-sum-no-permutation} are all the contributions due to $\xi_1, \ldots, \xi_p \in \Z^d\setminus \brag{0}$ such that, for some $j \in \brag{1,\ldots, p-1}$, one has $\sum_{i=1}^j \xi_i = 0$. By grouping them according to be the smallest such index $j$, that we denote by $q$, we bound from above the corresponding contributions in \eqref{eq:general-sum-no-permutation} as the product
\[ \begin{split} & T \sum_{\xi_1, \ldots, \xi_{q} \in Z^d\setminus\brag{0}}^\star \prod_{i=1}^{j} \frac{\exp\bra{- \epsilon |\xi_{i}|^2/2} }{|\xi_{i}|}  \prod_{j=1}^{q-1}  \frac{1}{| \sum_{i=1}^j \xi_{i} |^{1/H}} \delta_0\bra{\sum_{i=1}^q \xi_{i}} \cdot\\
&  \quad \cdot  T \sum_{\xi_{q+1}, \ldots, \xi_p \in \Z^d\setminus \brag{0}}   \prod_{i={q+1}}^p\frac{\exp\bra{- \epsilon |\xi_{i}|^2/2} }{|\xi_i|}  \prod_{j={q+1}}^{p-1} \min\brag{ \frac{1}{| \sum_{i=q+1}^j \xi_{i} |^{1/H}}, T} \delta_0\bra{\sum_{i=q+1}^p \xi_{i}}\\
& \quad \le C (T\sqrt{\epsilon})^{q} C (T\sqrt{\epsilon})^{p-q} 
\end{split} \]
where the last inequality follows from \eqref{eq:key-step} and the inductive assumption for the validity of \eqref{eq:general-sum-no-permutation} for every $p-q <p$.

\medskip
{\sc Step 3.} Proof of \eqref{eq:key-step}. Note that the conditions on the summation  yield $|\xi_i| \ge (|\xi_i|+ 1)/2$ as well as $|\sum_{i=1}^{j} \xi_i| \ge (|\sum_{i=1}^{j} \xi_i|+1)/2$, hence, for some constant $C = C(p)>0$, 
\[ \begin{split} & \sum_{\xi_1, \ldots, \xi_p}^\star   \prod_{i=1}^p\frac{\exp\bra{- \epsilon |\xi_i|^2/2} }{|\xi_i|}  \prod_{j=1}^{p-1}  \frac{1}{| \sum_{i=1}^j \xi_i |^{1/H}} \delta_0\bra{\sum_{i=1}^p \xi_i} \\
&  \quad \le C\sum_{\xi_1, \ldots, \xi_p \in \Z^d}   \prod_{i=1}^p\frac{\exp\bra{- \epsilon |\xi_i|^2/2} }{|\xi_i|+1}  \prod_{j=1}^{p-1}  \frac{1}{(| \sum_{i=1}^j \xi_i |+1)^{1/H}} \delta_0\bra{\sum_{i=1}^p \xi_i}\\ 
& \quad \le C\sum_{\xi_1, \ldots, \xi_p \in \Z^d}   \prod_{i=1}^p\frac{\exp\bra{- \epsilon |\xi_i|^2/C} }{|\xi_i|+1}  \prod_{j=1}^{p-1}  \frac{\exp\bra{- \epsilon |\sum_{i=1}^j \xi_i|^2/C} }{(| \sum_{i=1}^j \xi_i |+1)^{1/H}} \delta_0\bra{\sum_{i=1}^p \xi_i},\end{split}\]
where we also used the fact that $|\sum_{i=1}^j \xi_i|^2 \le j \sum_{i=1}^j |\xi_i|^2$, so that 
\[ \prod_{i=1}^p \exp\bra{- \epsilon |\xi_i|^2/2} \le\prod_{i=1}^p \exp\bra{- \epsilon |\xi_i|^2/C} \prod_{j=1}^{p-1} \exp\bra{- \epsilon \abs{ \sum_{i=1}^j \xi_i}^2/C} , \]
for some constant $C=C(p,H)>0$. We introduce the function  from Lemma~\ref{lem:young} (with $\epsilon/C$ instead of $\epsilon$)
\[ g(\xi) = \frac{ \exp\bra{ - \epsilon |\xi|^2/C}}{|\xi|+1},\]
so that the last line above can be rewritten, up to a constant $C>0$,
\begin{equation}\label{eq:ready-for-young} \begin{split}&  \sum_{\xi_1, \ldots, \xi_p \in \Z^d} \prod_{i=1}^p g(\xi_i) \prod_{j=1}^{p-1}  g \bra{ \sum_{i=1}^j \xi_i}^{1/H}  \delta_0\bra{\sum_{i=1}^p \xi_i}\\ 
& \quad =  \sum_{\xi_1, \ldots, \xi_{p-1} \in \Z^d} g(\xi_1)^{1+1/H} \bra{ \prod_{i=2}^{p-1} g(\xi_i)} \prod_{j=2}^{p-2}  g \bra{ \sum_{i=1}^j \xi_i}^{1/H} g\bra{\sum_{i=1}^{p-1} \xi_i}^{1+1/H},\end{split}\end{equation}
where we used the fact that  $\xi_p  = - \sum_{i=1}^{p-1} \xi_i$ and $g$ is even (notice also that the smallest  case we need to discuss is $p=3$, since $p=2$ is already covered by Proposition~\ref{prop:upper-bound-p-1-torus}).
We are now in a position to apply Lemma~\ref{lem:young} with $f_1=g^{1+1/H}$, $f_2 = \ldots = f_{p-1}=g$, $F_2 = \ldots = F_{p-2} = g^{1/H}$ and $F_{p-1} = g^{1+1/H}$. To this end we analyze the different cases separately and show that, choosing $\epsilon$ as in \eqref{eq:choice-epsilon}, \eqref{eq:ready-for-young} is bounded from above by $T^{-1} (T \sqrt{\epsilon})^p$, which in turn would imply \eqref{eq:general-sum-no-permutation}.

\medskip
{\sc Step 3.1}. Case study. 

{\sc Case $d=2+1/H$} We may choose the exponents
\[ 
\lambda_1 = 1, \quad \lambda_2 = d, \quad \Lambda_2 = \frac{d}{d-1},
\]
and
\[ 
\lambda_k = d, \quad \Lambda_k = \frac{d}{d-1} \quad \mbox{if}\; 3\le 
k\le p-1,
\]
so that \eqref{eq:lemyoungcond} is satisfied. Hence we obtain that \eqref{eq:ready-for-young} is bounded from above by the product
\begin{equation}\label{eq:final?} \norm{g^{1+1/H}}_{\ell^1} \bra{ \prod_{i=2}^{p-1} \norm{g}_{\ell^{d}} } \prod_{j=2}^{p-2} \norm{g^{1/H}}_{\ell^{d/(d-1)}} \norm{g^{1+1/H}}_{\ell^{d/(d-1)}}.\end{equation}
Using repeatedly Lemma~\ref{lem:g} to bound all these norms, we conclude that \eqref{eq:key-step} holds. Indeed, since $d=1/H + 2$, then $1+1/H = d-1$ and 
\begin{equation}
\begin{split}
 \norm{g^{1+1/H}}_{\ell^1} & = \norm{g}_{\ell^{1+1/H}}^{1+1/H} \lesssim \epsilon^{-\frac 1 2 \bra{ d/(1+1/H) - 1 }(1+1/H)} = \epsilon^{-\frac{1}{2}},\\
  \norm{g}_{\ell^d} & \lesssim |\log \epsilon|^{1/d},\\
 \norm{g^{1/H}}_{\ell^{d/(d-1)}} & = \norm{g}_{\ell^{d(d-2)/(d-1)}}^{d-2} \lesssim  \epsilon^{-\frac 1 2 \bra{ (d-1)/(d-2) - 1 }(d-2)} = \epsilon^{-\frac 1 2},\\
 \norm{g^{1+1/H}}_{\ell^{d/(d-1)}} & = \norm{g}_{\ell^{d(1+1/H)/(d-1)}}^{1+1/H} \lesssim |\log \epsilon|^{1-1/d},
 \end{split}
\end{equation}
so that, collecting all the terms \eqref{eq:final?} is bounded from above by
\[ \epsilon^{-\frac 12 (p-2)} |\log \epsilon|^{1+(p-3)/d} \lesssim  T^{-\frac 1 2 (p-2)} \bra{ \log T}^{ \frac {p-3} d + 2 - \frac p 2} \lesssim T^{-1} (T \sqrt{\epsilon})^{p}. \]

\medskip
{\sc  Case $d>2+1/H$}. We may choose the subcritical exponent $\bar{d} < d$ so that $1+1/H<\bar{d}-1<d$, $d(1+1/H)/(d-1)<d$ as well as $1/H<\bar{d}-2<d$. Hence by the choice of the exponents
\[ 
\lambda_1 = 1, \quad \lambda_2 = \bar{d}, \quad \Lambda_2 = \frac{\bar{d}}{\bar{d}-1},
\]
and
\[ 
\lambda_k = \bar{d}, \quad \Lambda_k = \frac{\bar{d}}{\bar{d}-1} \quad \mbox{if}\; 3\le 
k\le p-1,
\]
\eqref{eq:lemyoungcond} is again satisfied. Thus we can bound from above \eqref{eq:ready-for-young} by the product
\begin{equation}\label{eq:final?d>1/H+2} \norm{g^{1+1/H}}_{\ell^1} \bra{ \prod_{i=2}^{p-1} \norm{g}_{\ell^{\bar{d}}} } \prod_{j=2}^{p-2} \norm{g^{1/H}}_{\ell^{\bar{d}/(\bar{d}-1)}} \norm{g^{1+1/H}}_{\ell^{\bar{d}/(\bar{d}-1)}}.
\end{equation}
Again using repeatedly Lemma~\ref{lem:g} we can bound all these norms. Indeed, considering them separately yields
\begin{equation}
\begin{split}
 \norm{g^{1+1/H}}_{\ell^1} & = \norm{g}_{\ell^{1+1/H}}^{1+1/H} \lesssim \epsilon^{-\frac 1 2 \bra{ d/(1+1/H) - 1 }(1+1/H)} = \epsilon^{-\frac{1}{2}\bra{ d-1-1/H}},\\
  \norm{g}_{\ell^{\bar{d}}} & \lesssim \epsilon^{- \frac12 \bra{d/\bar{d}-1}},\\
 \norm{g^{1/H}}_{\ell^{\bar{d}/(\bar{d}-1)}} & = \norm{g}_{\ell^{\bar{d}/(H(\bar{d}-1))}}^{1/H} \lesssim  \epsilon^{-\frac 1 2 \bra{ Hd(\bar{d}-1)/\bar{d} - 1 }/H} = \epsilon^{-\frac 1 2\bra{d\bra{1-1/\bar{d}}-1/H}},\\ 
 \norm{g^{1+1/H}}_{\ell^{\bar{d}/(\bar{d}-1)}} & = \norm{g}_{\ell^{\bar{d}(1+1/H)/(\bar{d}-1)}}^{1+1/H} \lesssim \epsilon^{-\frac 1 2 \bra{ d(\bar{d}-1)/(\bar{d}(1+1/H)) - 1 }(1+1/H)} \\
 & = \epsilon^{-\frac 1 2\bra{d\bra{1-1/\bar{d}}-1-1/H}},
 \end{split}
\end{equation}
so that, collecting all the terms \eqref{eq:final?d>1/H+2} is bounded from above by
\[ \epsilon^{-\frac 12 \bra{(p-1)(d-1/H) - p}} \lesssim  T^{p-1} \cdot T^{- \frac{p}{d-1/H}} \sim T^{-1} \bra{T \sqrt{\epsilon}}^p. \]

\medskip
{\sc Case $d<2+1/H$}. We argue that there exists $1<\bar{d}<d$ such that 
\begin{equation}\label{eq:choicebardsmalldim}
\frac{\bar{d}}{\bar{d}-1}\bra{1+\frac1H}>\frac{\bar{d}}{\bar{d}-1}\frac1H>d.
\end{equation}
Indeed, the first inequality is trivial since $1+1/H>1/H$, while the second one is always satisfied if $d \le 1/H$. 
If $d>1/H$ the second inequality is equivalent to 
$
\bar{d} <  d/ (d-1/H).
$
Moreover, note that since $d<2+1/H$ we have
$
 d/ (d-1/H) >  d/2,
$
thus choosing $\bar d$ such that 
\begin{align*}
& \max\brag{1,\frac d 2} < \bar{d} < d & \text{if $d\le\frac1H$}\\
& \max\brag{1,\frac d 2} < \bar{d} < \frac d {d-\frac1H} & \text{if $d>\frac1H$}
\end{align*}
\eqref{eq:choicebardsmalldim} is satisfied. Given $\bar{d}$ satisfying \eqref{eq:choicebardsmalldim}, to bound \eqref{eq:ready-for-young} we may choose the exponents 
\[ 
\lambda_1 = 1, \quad \lambda_2 = \bar{d}, \quad \Lambda_2 = \frac{\bar{d}}{\bar{d}-1},
\]
and
\[ 
\lambda_k = \bar{d}, \quad \Lambda_k = \frac{\bar{d}}{\bar{d}-1} \quad \mbox{if}\; 3\le 
k\le p-1,
\]
so that again \eqref{eq:lemyoungcond} is satisfied. Thus \eqref{eq:ready-for-young} is bounded from above by the product
\begin{equation}\label{eq:final?d<1/H+2} \norm{g^{1+1/H}}_{\ell^1} \bra{ \prod_{i=2}^{p-1} \norm{g}_{\ell^{\bar{d}}} } \prod_{j=2}^{p-2} \norm{g^{1/H}}_{\ell^{\bar{d}/(\bar{d}-1)}} \norm{g^{1+1/H}}_{\ell^{\bar{d}/(\bar{d}-1)}}.
\end{equation}
In this case we need to do an additional distinction between the cases $d<1+1/H, d=1+1/H$ and $1+1/H<d<2+1/H$. Indeed the first factor of \eqref{eq:final?d<1/H+2} might give a diverging term depending on $d$. 

If $d<1+1/H$ an application of Lemma~\ref{lem:g} yields the upper bounds
\begin{equation*}
\begin{split}
 \norm{g^{1+1/H}}_{\ell^1} & = \norm{g}_{\ell^{1+1/H}}^{1+1/H} \lesssim 1,\\
  \norm{g}_{\ell^{\bar{d}}} & \lesssim \epsilon^{- \frac12 \bra{d/\bar{d}-1}},\\
 \norm{g^{1/H}}_{\ell^{\bar{d}/(\bar{d}-1)}} & = \norm{g}_{\ell^{\bar{d}/(H(\bar{d}-1))}}^{1/H} \lesssim  1,\\ 
 \norm{g^{1+1/H}}_{\ell^{\bar{d}/(\bar{d}-1)}} & = \norm{g}_{\ell^{\bar{d}(1+1/H)/(\bar{d}-1)}}^{1+1/H} \lesssim 1,
 \end{split}
\end{equation*}
so that, collecting all the terms \eqref{eq:final?d<1/H+2} is bounded from above by
\[ 
\epsilon^{-\frac 12 \bra{\bra{\frac d{\bar{d}}-1}\bra{p-2}}} \lesssim  T^{\frac 12 \bra{\bra{\frac d{\bar{d}}-1}\bra{p-2}}} \lesssim T^{\frac p 2 - 1}\sim T^{-1} \bra{T \sqrt{\epsilon}}^p. 
\]

If $d>1+1/H$ an application of Lemma~\ref{lem:g} yields the upper bounds
\[
\begin{split}
 \norm{g^{1+1/H}}_{\ell^1} & = \norm{g}_{\ell^{1+1/H}}^{1+1/H} \lesssim \epsilon^{-\frac12\bra{d-1-\frac1H}},\\
  \norm{g}_{\ell^{\bar{d}}} & \lesssim \epsilon^{- \frac12 \bra{d/\bar{d}-1}},\\
 \norm{g^{1/H}}_{\ell^{\bar{d}/(\bar{d}-1)}} & = \norm{g}_{\ell^{\bar{d}/(H(\bar{d}-1))}}^{1/H} \lesssim  1,\\ 
 \norm{g^{1+1/H}}_{\ell^{\bar{d}/(\bar{d}-1)}} & = \norm{g}_{\ell^{\bar{d}(1+1/H)/(\bar{d}-1)}}^{1+1/H} \lesssim 1,
 \end{split}
\]
so that collecting all the terms \eqref{eq:final?d<1/H+2} is bounded from above by
\[ 
\epsilon^{-\frac 12 \bra{\bra{\frac d{\bar{d}}-1}\bra{p-2}+ d - 1-\frac1H}} \lesssim  T^{\frac 12 \bra{\bra{\frac d{\bar{d}}-1}\bra{p-2}+d-1-\frac1H}} \lesssim T^{\frac p 2 - 1}\sim T^{-1} \bra{T \sqrt{\epsilon}}^p, 
\]
where the last inequality is true if
\[
\bra{\frac d{\bar{d}}-1}\bra{p-2}+d-1-\frac1H < p-2,
\]
which is satisfied if $d>1+1/H$.

Finally if $d=1+1/H$ an application of Lemma~\ref{lem:g} yields the upper bounds
\[
\begin{split}
 \norm{g^{1+1/H}}_{\ell^1} & = \norm{g}_{\ell^{1+1/H}}^{1+1/H} \lesssim \log \epsilon,\\
  \norm{g}_{\ell^{\bar{d}}} & \lesssim \epsilon^{- \frac12 \bra{d/\bar{d}-1}},\\
 \norm{g^{1/H}}_{\ell^{\bar{d}/(\bar{d}-1)}} & = \norm{g}_{\ell^{\bar{d}/(H(\bar{d}-1))}}^{1/H} \lesssim  1,\\ 
 \norm{g^{1+1/H}}_{\ell^{\bar{d}/(\bar{d}-1)}} & = \norm{g}_{\ell^{\bar{d}(1+1/H)/(\bar{d}-1)}}^{1+1/H} \lesssim 1,
 \end{split}
\]
so that collecting all the terms \eqref{eq:final?d<1/H+2} is bounded from above by
\[ 
\epsilon^{-\frac 12 {\bra{\frac d{\bar{d}}-1}\bra{p-2}}} \log \epsilon \lesssim  T^{\frac 12 {\bra{\frac d{\bar{d}}-1}\bra{p-2}}} \lesssim T^{\frac p 2 - 1}\sim T^{-1} \bra{T \sqrt{\epsilon}}^p, 
\]
where the last inequality is satisfied if $\bar{d} >  d /2.$
\end{proof}

\subsection{Conclusion}

We complete the proof of Theorem~\ref{thm:main-Td} with the next proposition.

\begin{proposition}\label{prop:lower-bound}
Let $\epsilon =  \epsilon(d,H,T)>0$ be as in \eqref{eq:choice-epsilon}. Then, for every $T \ge 0$ sufficiently large,
\begin{equation}\label{eq:lower-W1-empirical}
\Ex\braq{ W_1( \mu_T, T) } \gtrsim  T \sqrt{\epsilon},
\end{equation}
and, for every $p \ge 1$,
\begin{equation}\label{eq:upper-Wp-empirical}
\Ex\braq{ W_p^p( \mu_T, T) } \lesssim C T (\sqrt{\epsilon})^p.
\end{equation}
\end{proposition}
\begin{proof}
By \eqref{eq:lower-p} we may write
\[
\begin{split}
\Ex\braq{ W_1( \mu_T, T) } & \ge \frac1M \Ex \braq{\| \nabla u_\epsilon \|_{L^2}^2} - \frac C {M^3} \Ex \braq{\| \nabla u_\epsilon\|_{L^4}^4} \\
& \stackrel{\eqref{eq:asymptotics-L2-gradient}, \eqref{eq:p-moment-gradient}}{\gtrsim}  \frac{1}{M} \bra{T \sqrt{\epsilon}}^2 - \frac{C}{M^3} \bra{T \sqrt{\epsilon}}^4.
\end{split}
\]
Finally, choosing $M \sim T \sqrt{\epsilon}$ yields \eqref{eq:lower-W1-empirical}.

\medskip
Let us turn to \eqref{eq:upper-Wp-empirical}. By the H\"older inequality and  \eqref{eq:monotonicity}, it is sufficient to argue in the case $p\in \mathbb{N}$ even, so that we may use \eqref{eq:upper-p} and \eqref{eq:p-moment-gradient} to obtain
\[
\begin{split}
\Ex\braq{ W_p^p( \mu_T, T) } & \lesssim T \bra{\sqrt{\epsilon}}^p + T^{1-p} \| \nabla u_\varepsilon\|_{L^p}^p \lesssim T \bra{\sqrt{\epsilon}}^p,
\end{split}
\]
which concludes the proof.
\end{proof}

\section{A lower bound for fBm's on $\R^d$}\label{sec:corollary}

Our arguments presently do not follow through in the non-compact case $\R^d$, except if heavily modified, e.g., by localizing on a large ball as $T$ increases, and anyway they yield asymptotic bounds that do not seem to be sharp. However, Theorem~\ref{thm:main-Td} combined with a convexity argument gives a straightforward lower bound for the expected distance between the empirical measures of two independent fBm's on $\R^d$. 

\begin{corollary}\label{cor:two-paths}
Let $B^{H_1} = \bra{B_t^{H_1}}_{t \ge 0}$, $B^{H_2}= \bra{B_t^{H_2}}_{t \ge 0}$ be independent fractional Brownian motions on $\R^d$, with Hurst indexes $H_1, H_2 \in \bra{0,1}$ and assume that $H_1\ge H_2$. Then, for every $p\ge 1$, as $T \rightarrow \infty$,
\begin{equation}\label{eq:rates-corollary}
 \Ex \braq{W_p^p \bra{\int_0^T \delta_{B^{H_1}_s} \,\d s , \int_0^T \delta_{B^{H_2}_s} \,\d s}} \gtrsim T  \cdot  \begin{cases}
T^{-p/2} & \text{if}\ d < \frac{1}{H_1} + 2, \\
\bra{\log T/T}^{p/2} & \text{if}\ d = \frac{1}{H_1} + 2, \\
T^{-\frac{p}{d-1/H_1} } & \text{if}\ d > \frac{1}{H_1} + 2.
\end{cases} 
\end{equation}
\end{corollary}

\begin{proof}
Denote by $\Ex_1\braq{ \cdot }$, $\Ex_{2}\braq{\cdot }$ respectively expectation with respect to $B^{H_1}$ and $B^{H_2}$, so that $\Ex\braq{\cdot} = \Ex_1 \braq{ \Ex_2\braq{\cdot}}$. Notice first that, by convexity, 
\[ \begin{split} \Ex_2\braq{ W_p^p\bra{ \int_0^T \delta_{B^{H_1}_s} \,\d s , \int_0^T \delta_{B^{H_2}_s} \,\d s}} & \ge W_p^p \bra{ \int_0^T \delta_{B^{H_1}_s} \,\d s , \int_0^T \Ex_2 \braq{ \delta_{B^{H_2}_s}} \,\d s}\\
& = W_p^p \bra{ \int_0^T \delta_{B^{H_1}_s} \,\d s , \int_0^T \mathcal{N}(0, s^{2H_2}) \,\d s}.
\end{split}\]
Denote by $\pr: \R^d \to \T^d$ the projection map, which is $1$-Lipschitz,
$$ \dT( \pr(x), \pr(y) ) \le |x-y|,$$
so that, for measures $\mu$, $\nu$ on $\R^d$, it holds
$$  W_{p}(  \mu, \nu) \ge W_{p,\T^d}( \pr_\sharp \mu, \pr_\sharp \nu),$$
where $\pr_\sharp \mu (A) = \mu( \pr^{-1}(A))$ denotes the push-forward of $\mu$ by $\pr$. We apply it to $\mu =   \int_0^T \delta_{B^{H_1}_s}\,\d s$, for which $\pr_\sharp \mu = \int_0^T \delta_{\pr(B^{H_1}_s)}\,\d s$, and $\nu = \int_0^T \mathcal{N}(0, s^{2H_2}) \,\d s$, hence
$$ W_p^p \bra{ \int_0^T \delta_{B^{H_1}_s} \,\d s , \int_0^T \mathcal{N}(0, s^{2H_2}) \,\d s} \ge W_{p, \T^d}^p \bra{ \int_0^T \delta_{\pr(B^{H_1}_s)} \,\d s , \int_0^T \pr_\sharp \mathcal{N}(0, s^{2H_2}) \,\d s},$$
where $W_{p,\mathbb{T}^d}$ denotes the Wasserstein distance on the torus $\mathbb{T}^d$. We then use the triangle inequality combined with the inequality $(x+y)^p\leq 2^{p-1}(x^p+y^p)$, to bound
\[ \begin{split}  W^p_{p,\T^d}\bra{ \int_0^T \delta_{\pr(B^{H_1}_s)} \,\d s, \int_0^T \pr_\sharp \mathcal{N}(0, s^{2H_2}) \,\d s} & \gtrsim    W_{p, \T^d}^p \bra{ \int_0^T \delta_{\pr(B^{H_1}_s)} \,\d s , T } \\
&  \quad -  W_{p,\T^d}^p\bra{ \int_0^T\pr_\sharp \mathcal{N}(0, s^{2H_2})\,\d s, T} .\end{split}\]
We claim that, writing $q = \min\brag{2, p/(p-1)}$, one has the bounds:
\begin{equation}\label{eq:claim-small}  W_{p,\T^d}^p\bra{ \int_0^T\pr_\sharp \mathcal{N}(0, s^{2H_2})\,\d s, T }  \lesssim T  \cdot  \begin{cases}
T^{-p} & \text{if}\ d < q\bra{ \frac{1}{H_2} + 1}, \\
(\log T)^{p-1}/T^p & \text{if}\ d = q\bra{ \frac{1}{H_2} + 1}, \\
T^{-\frac{p}{d/q-1/H_2} } & \text{if}\ d > q \bra{ \frac{1}{H_2} + 1}.
\end{cases}  \end{equation}
In particular, the right hand side above is always infinitesimal
with respect to that in \eqref{eq:rates-corollary}, therefore, taking expectation with respect to  $\Ex_1\braq{\cdot}$ the thesis follows from the lower bounds in \eqref{eq:rates-main-theorem}. 

\medskip
To prove \eqref{eq:claim-small}, notice first that by monotonicity \eqref{eq:monotonicity}, we can assume that $p\ge 2$, so that $q = p/p-1$ is the dual exponent.
The thesis follows from an application of \eqref{eq:wasserstein-p>1} from Lemma~\ref{lem:general-transport-pde-bound} with $\mu = \int_0^T\pr_\sharp \mathcal{N}(0, s^{2H_2})\,\d s$, $\mu(\T^d) = T = \hat{\mu}(0)$. For every $\xi \in \Z^d\setminus \brag{0}$, 
$$ \hat{\mu}(\xi) = \int_0^T \exp\bra{ -2 \pi^2 |\xi|^2 s^{2H_2}} \,\d s \lesssim \frac{1}{|\xi|^{1/H_2}}.$$
Therefore, letting $u_\epsilon$ be the solution to the Poisson equation $-\Delta u_\epsilon = P_\epsilon(\mu - T)$, by \eqref{eq:gradient-fourier} we have for $\xi \in \Z^d$,
\[ |\hat{u}_{\epsilon}(\xi)| \lesssim \frac{\exp\bra{ - 2 \pi^2 \epsilon |\xi|^2}}{(|\xi|+1)^{2+1/H_2}}\]
and the Hausdorff-Young inequality \eqref{eq:hausdorff-young} entails
%
\[ \norm{\nabla u_\epsilon}_{L^p(\T^d)}^p \lesssim \norm{ g}_{L^{q(1+1/H_2)}(\Z^d)}^{p(1+1/H_2)},\]
where we introduce the function  from Lemma~\ref{lem:g} (with $\epsilon/C$ for a suitable constant $C=C(d,p, H)>0$ instead of $\epsilon$),
\[ g(\xi) = \frac{ \exp\bra{ - \epsilon |\xi|^2/C}}{|\xi|+1}.\]
It follows that
\[\begin{split} &  W_{p,\T^d}^p\bra{ \int_0^T\pr_\sharp \mathcal{N}(0, s^{2H_2})\,\d s, T } \\
& \quad \lesssim T (\sqrt{\epsilon})^p+  T^{1-p}  \cdot  \begin{cases}
1 & \text{if}\ d < q\bra{ \frac{1}{H_2} + 1}, \\
|\log \epsilon|^{p-1} & \text{if}\ d = q\bra{ \frac{1}{H_2} + 1}, \\
\epsilon^{-\frac 1 2(d/q - 1 - 1/H)p} & \text{if}\ d > q \bra{ \frac{1}{H_2} + 1}.
\end{cases}
\end{split}
\] 
Letting $\sqrt{\epsilon} = 1/T$ in the first case and $\sqrt{\epsilon} = T^{-\frac{1}{d/q-1/H_2}}$ in the second and third case, we obtain \eqref{eq:claim-small}.
%
\end{proof}

\section{A result on discrete-time approximation }\label{sec:thm:discrete}

We prove the following variant of Theorem~\ref{thm:main-Td} where we consider a discrete-time approximation of fBm. We limit ourselves to the case $p=1$ for simplicity. 

\begin{theorem}\label{thm:main-discrete}
Let $B^H = \bra{B_t^H}_{t \ge 0}$ be a fractional Brownian motion with Hurst index $H \in \bra{0,1}$ taking values on a $d$-dimensional torus $\mathbb{T}^d$. Let $\alpha > 0$ and for $T \ge 0$, set $\tau \sim  T^{-\alpha}$. Then, as $T \rightarrow \infty$,
\[
 \Ex \braq{W_1 \bra{ \sum_{t = 1}^{\lfloor T/\tau \rfloor} \delta_{B^H_{t \tau}} \tau, \lfloor T/\tau \rfloor \tau }} \lesssim T  \cdot  \begin{cases}
  T^{-1/2} &  \text{if $d \le   2$,} \\
  T^{- \min\{ 1/2, \frac{1+\alpha}{d}\}}  & \text{if $2< d< \frac{1}{H} + 2$,} \\
  \max\{ \sqrt{{\log T}/T},   T^{-  \frac{1+\alpha}{d}}\} & \text{if $d = \frac{1}{H} + 2$,}\\
   T^{-\min\{ (d-1/H), \frac{1+\alpha}{d}\}} &  \text{if $d> \frac1 H +2$.}
\end{cases} 
\]
\end{theorem}

The argument follows a similar path as in the previous section. 

\subsection{Fourier transform moment bounds} In this case, since we limit ourselves to the Wasserstein distance of order $1$, only bounds for second and fourth moments will be needed. The following result is a generalization of Lemma~\ref{lem:second-moment-continuous}.

\begin{lemma}\label{lem:second-moment-discrete}
For $T \ge \tau$  it holds, for every $\xi \in \Z^d$, $\xi\neq 0$,
\begin{equation}\label{eq:lower-bound-fourier-modes-discrete}
\Ex\braq{ \abs{ \widehat{\mu_{\tau,T}}(\xi)}^2 } \sim T  \bra{ |\xi|^{-1/H} + \tau}   
 \end{equation}
\end{lemma}

\begin{proof}
We use the inequality, valid for any absolutely continuous decreasing function $f$,
\[  \abs{ \sum_{s=0}^{m-1} f(s \tau )  \tau - \int_0^{m \tau }f(t) \,\d t  } \le  \tau \int_0^{m \tau} |f'(t)| \,\d t \le \tau \bra{ f(0)-f(m \tau)} \]
that for $f(x) = \exp\bra{ - |\xi|^2 x^{2H}}$ yields
\[ \abs{ \sum_{s=0}^{m-1}  \exp\bra{ -  |\xi|^2 s^{2H} \tau^{2H} } \tau - \int_0^{m \tau} \exp\bra{ -  |\xi|^2 t^{2H} } \,\d t } \le \tau \bra{ 1 - \exp\bra{ - |\xi|^2 (m\tau)^{2H}}}.\]
We rewrite
$$ \int_0^{m \tau} \exp\bra{ -  |\xi|^2 t^{2H} } \,\d t  = |\xi|^{-1/H} \int_0^{m\tau |\xi|^{1/H}} \exp\bra{ - t^{2H}} \,\d t $$
In our application we have $m\tau \ge 1$, hence 
$$ \int_0^{m\tau |\xi|^{1/H}} \exp\bra{ - t^{2H}} \,\d t \sim 1 \quad \text{and} \quad   \bra{ 1 - \exp\bra{ - |\xi|^2 (m\tau)^{2H}}} \sim 1,$$
so that
$$ \sum_{s=0}^{m-1}  \exp\bra{ -  |\xi|^2 s^{2H} \tau^{2H} } \tau  \sim |\xi|^{-1/H} + \tau.$$

Let then $n= \lfloor T/\tau \rfloor$ and write
\[ \begin{split}
 \Ex\braq{ \abs{  \widehat{\mu_{\tau,T}}(\xi) }^2 } & = \sum_{s,t=1}^{n} \Ex\braq{ \exp\bra{ 2 \pi  i \xi \cdot (B_{t\tau} - B_{s \tau})}}  \tau^2  \\
 & = \sum_{s,t=1}^{n}  \exp\bra{ -  2\pi^2|\xi|^2|t-s|^{2H} \tau^{2H}} \tau^2 \\
& \le 2 n \tau \sum_{s=0}^{n-1} \exp\bra{ -  2\pi^2| \xi|^2 s^{2H} \tau^{2H}} \tau \\
& \lesssim T  \bra{  |\xi|^{-1/H} + \tau }
\end{split}\]
For the lower bound, we write instead
\[ \begin{split}
\Ex\braq{ \abs{  \widehat{\mu_{\tau,T}}(\xi) }^2 } & =  \sum_{s,t=1}^{n}  \exp\bra{ -  2\pi^2| \xi|^2|t-s|^{2H} \tau^{2H}}  \tau^2   \\
& \ge \sum_{t=\lfloor n/2 \rfloor}^n \sum_{s=t}^{n} \exp\bra{ -  2\pi^2| \xi|^2|t-s|^{2H} \tau^{2H}}  \tau^2\\
& \ge \lfloor n/2 \rfloor \tau \sum_{s=0}^{\lfloor n/2 \rfloor} \exp\bra{ -  2\pi^2| \xi|^2 s^{2H} \tau^{2H}}  \tau,
\end{split}\]
from which we argue similarly as in the upper bound, and obtain the thesis.
\end{proof}

\begin{lemma}\label{lem:upper-bound-p-discrete}
For every $p \in \mathbb{N}$, 
for every $T \ge \tau$ and $\xi \in (\Z^d)^p$,
\begin{equation}\label{eq:upper-bound-fourier-modes-discrete}
\abs{ \Ex\braq{ \prod_{j=1}^p\widehat{\mu_{\tau, T}}(\xi_j) }} \lesssim  \sum_{\sigma \in \mathcal{S}_p} \prod_{j=1}^{p} \min\brag{ \frac{1}{| \sum_{i=1}^j \xi_{\sigma(i)} |^{1/H}} + \tau , T}.
\end{equation} 
\end{lemma}


\begin{proof}
Let $n= \floor{T/\tau}$ and write by definition,
\[ \widehat{\mu_{\tau,T}}(\xi_j) = \sum_{t_1,\dots,t_p=1}^n \exp\bra{ 2 \pi i \xi B_{t_j\tau}} \tau \]
so that
\[ \prod_{i=1}^p \widehat{\mu_T}(\xi_i) = \sum_{j=1}^n \exp\bra{ 2 \pi i \sum_{j=1}^p \xi_j B_{t_j \tau}} \tau^p .\]
As in the proof of Lemma~\ref{lem:upper-bound-p} we split the summation into $p!$ simplexes, one for every $\sigma \in \mathcal{S}_p$,
\[ \Delta_\sigma := \brag{ 1 \le t_{\sigma(1)} \le \ldots  \le t_{\sigma(p)} \le n}.\]
We now argue only in the case $\sigma$ being the identity permutation, the other cases being analogous. A summation by parts gives 
\[ \sum_{j=1}^p \xi_j B_{t_j \tau} =   B_{0} \sum_{i=1}^p \xi_i + \sum_{j=1}^{p}  (B_{t_j \tau}- B_{t_{j-1}\tau}) \sum_{i=j}^{p} \xi_i = \sum_{j=1}^{p}  (B_{t_j\tau}- B_{t_{j-1}\tau}) \sum_{i=j}^{p} \xi_i,\]
where we let $t_0 = 0$ and we assume that $B_0= 0$. Using this identity, the Fourier transform (characteristic function) of a Gaussian random variable and \eqref{eq:local-non-determinism}, it follows that
\[ \Ex\braq{ \exp\bra{ i \sum_{j=1}^p \xi_j B_{t_j\tau}}} \le \exp\bra{- \frac C 2 \sum_{j=1}^{p} \abs{ \sum_{i=j}^{p} \xi_i  }^2 |t_{j}- t_{j-1}|^{2H}\tau^{2H}}.\]
We then bound from above the sum
\[ \begin{split}  \sum_{\Delta_\sigma} \Ex\braq{ \exp\bra{ i \sum_{j=1}^p \xi_j B_{t_j\tau}}} \tau^p
 & \le  \sum_{\Delta_\sigma}  \exp\bra{- \frac C 2 \sum_{j=1}^{p} \abs{ \sum_{i=j}^{p} \xi_i  }^2 |t_{j}- t_{j-1}|^{2H}\tau^{2H}} \tau^p \\
& \le \sum_{t_1,\dots,t_p=1}^n \exp\bra{- \frac C 2 \sum_{j=1}^{p} \abs{ \sum_{i=j}^{p} \xi_i  }^2 s_j^{2H}\tau^{2H}} \tau^p,
\end{split}\]
where we performed the change of variables $s_j= t_j- t_{j-1}$, for $j \ge 1$, recalling that $t_0=0$. To conclude, we split into a product of $p$ sums that we bound separately
\[ \sum_{s_j=1}^n \exp\bra{- \frac C 2  \abs{ \sum_{i=j}^{p} \xi_i  }^2 s_j^{2H}}  \tau \lesssim \min\brag{  \abs{ \sum_{i=j}^{p} \xi_i}^{-1/H}+\tau, T}.\]
\end{proof}

\subsection{Upper bound}

\begin{proposition}\label{prop:upper-bound-p-1-torus-discrete}
Define $\epsilon = \epsilon(d,H,\alpha, T)>0$ as follows:
\begin{equation}\label{eq:choice-epsilon-discrete} \sqrt{\epsilon} = \begin{cases}
T^{-1/2} & \text{if $d \le 2$,}\\
T^{-1/2} & \text{if $2<d<2+\frac1H$ and $\alpha > d/2 -1$,}\\
T^{-\frac{1+\alpha}{d}} & \text{if $2<d<2+\frac1H$ and $\alpha \le d/2 -1$,}\\
T^{-\frac{1+\alpha}{d}} & \text{if $d=2+\frac1H$ and $\alpha < \frac1{2H}$,} \\
\sqrt{{\log T}/T} & \text{if $d=2+\frac1H$ and $\alpha \ge \frac1{2H}$,}\\
T^{-\frac{1+\alpha}{d}} & \text{if $d>2+\frac1H$ and $\alpha \le \frac{1/H}{d-1/H}$,}\\
T^{-\frac{1}{d-1/H}} & \text{if $d>2+\frac1H$ and $\alpha > \frac{1/H}{d-1/H}$,}
\end{cases}
\end{equation}
and let $u_\epsilon$ be a solution to the PDE
\[
- \Delta u_\epsilon = P_\epsilon (\mu_{\tau, T} - \floor{T/\tau}\tau).
\]
Then, as $ T \to \infty$,
\begin{equation}\label{eq:asymptotics-L2-gradient-discrete}
\Ex\braq{\norm{ \nabla u_\epsilon }_{L^2}} \lesssim  T \sqrt{\epsilon},
\end{equation}
hence, by \eqref{eq:upper-p},
\[ \Ex\braq{W_1(\mu_{\tau, T}, \floor{T/\tau}\tau) } \lesssim T \sqrt{\epsilon}. \]
\end{proposition}

\begin{proof}
{\sc Step 1}. We argue that 
\begin{equation}\label{eq:step1-second-moment-discrete}
\Ex\braq{\norm{ \nabla u_{\epsilon} }_{L^2}^2} \sim T \norm{ g }_{\ell^{2 +1/H}(\Z^d)}^{2+1/H} + T^{1-\alpha} \norm{ g }_{\ell^{2}(\Z^d)}^{2},
\end{equation}
where the first $g$ in the sum above is as in Lemma~\ref{lem:g}  but with $\epsilon/(2+1/H)$ instead of $\epsilon$. Indeed, by Plancherel's identity (see \eqref{eq:plancherelidentity}) and Lemma~\ref{lem:second-moment-discrete}, we may estimate the second moment of $\nabla u_\epsilon$ by
\[
\begin{split} \Ex\braq{\norm{ \nabla u_{\epsilon} }_{L^2}^2} & = (2 \pi)^{-2} \sum_{\xi \in \Z^d \setminus\brag{0}} \Ex\braq{\abs{\widehat{\mu_{\tau, T}}(\xi)}^2} \frac{\exp\bra{-\epsilon|\xi|^2/2 }}{|k|^2}\\
& \sim   T  \sum_{\xi \in \Z^d \setminus\brag{0}} \frac{\exp\bra{-\epsilon|\xi|^2/2}}{|\xi|^{2}} \bra{ |\xi|^{-1/H} + \tau }.\\
& \sim T \norm{ g }_{\ell^{2 +1/H}(\Z^d)}^{2+1/H} + T^{1-\alpha} \norm{ g }_{\ell^{2}(\Z^d)}^{2}.
\end{split}
\]

{\sc Step 2} Case study. We now argue by Lemma \ref{lem:g} to show that \eqref{eq:asymptotics-L2-gradient-discrete} holds with $\epsilon$ chosen as in \eqref{eq:choice-epsilon-discrete}. 

\medskip
{\sc Case $d<2$}. By Lemma \ref{lem:g} the right hand side of \eqref{eq:step1-second-moment-discrete} is $\sim T$ and thus \eqref{eq:asymptotics-L2-gradient-discrete}. If $d=2$, by Lemma \ref{lem:g} the first term on the right hand side of \eqref{eq:step1-second-moment-discrete} is $\sim T$, while the second term is $\sim T^{1-\alpha} | \log \epsilon |$. Hence, by choosing $\epsilon = T^{-1}$ as in \eqref{eq:choice-epsilon-discrete}, \eqref{eq:asymptotics-L2-gradient-discrete} holds.  

\medskip
{\sc Case $2< d < 2+  1 /H$}. By Lemma \ref{lem:g} the right hand side of \eqref{eq:step1-second-moment-discrete} is $\sim T + T^{1-\alpha} \epsilon^{-\frac 1 2 (d-2) }$, which in turns imply \eqref{eq:asymptotics-L2-gradient-discrete}. Indeed, the condition $\alpha \le d/2-1$ is equivalent to 
$$ T^2 \varepsilon = T^{1-\alpha} \epsilon^{-\frac 1 2 (d-2) },$$
which implies \eqref{eq:asymptotics-L2-gradient-discrete}, otherwise the first term on the right hand side of \eqref{eq:step1-second-moment-discrete} is the leading one and \eqref{eq:asymptotics-L2-gradient-discrete} is still satisfied choosing $\sqrt{\epsilon} = T^{-1/2}$.

\medskip
{\sc Case $d = 2 + 1/H$}. By Lemma \ref{lem:g} the right hand side of \eqref{eq:step1-second-moment-discrete} is $\sim T | \log T| + T^{1-\alpha} \epsilon^{-\frac12 (d-2)}$. Note that if $\alpha < 1/2H$ by the choice $\sqrt{\epsilon} = T^{-(1+\alpha)/d}$ the second term of \eqref{eq:step1-second-moment-discrete} is the leading term. Otherwise we may choose $\sqrt{\epsilon} = \sqrt{\log T/T}$ and \eqref{eq:asymptotics-L2-gradient-discrete} holds. 

\medskip
{\sc Case $d > 2 + 1/H$}. Again by Lemma \ref{lem:g} the right hand side of \eqref{eq:step1-second-moment-discrete} is $\sim T \epsilon^{-\frac12 (d-2-1/H)}$ $+ T^{1-\alpha} \epsilon^{-\frac12 (d-2)}$. If $\alpha \leq (1/H)/(d-1/H)$ we may choose $\sqrt{\epsilon} = T^{-(1+\alpha)/d}$ so that the second term in \eqref{eq:step1-second-moment-discrete} is the leading one and \eqref{eq:asymptotics-L2-gradient-discrete} holds. Otherwise we might choose $\sqrt{\epsilon} = T^{-1/(d-1/H)}$ so that the first term of \eqref{eq:step1-second-moment-discrete} is the leading one and \eqref{eq:asymptotics-L2-gradient-discrete} holds. 
\end{proof}

\bibliographystyle{plain}
\bibliography{biblio.bib}

\begin{thebibliography}{10}

\bibitem{ajtai1984optimal}
Mikl{\'o}s Ajtai, J{\'a}nos Koml{\'o}s, and G{\'a}bor Tusn{\'a}dy.
\newblock On optimal matchings.
\newblock {\em Combinatorica}, 4(4):259--264, 1984.

\bibitem{ambrosio2005gradient}
Luigi Ambrosio, Nicola Gigli, and Giuseppe Savar{\'e}.
\newblock {\em Gradient flows: in metric spaces and in the space of probability
  measures}.
\newblock Springer Science \& Business Media, 2005.

\bibitem{ambrosio2019pde}
Luigi Ambrosio, Federico Stra, and Dario Trevisan.
\newblock A pde approach to a 2-dimensional matching problem.
\newblock {\em Probability Theory and Related Fields}, 173(1):433--477, 2019.

\bibitem{biagini_stochastic_2008}
F.~Biagini, Y.~Hu, B.~{\O}ksendal, and T.~Zhang.
\newblock {\em Stochastic {Calculus} for {Fractional} {Brownian} {Motion} and
  {Applications}}.
\newblock Probability and {Its} {Applications}. Springer London, 2008.

\bibitem{bobkov2021simple}
Sergey~G Bobkov and Michel Ledoux.
\newblock A simple fourier analytic proof of the akt optimal matching theorem.
\newblock {\em The Annals of Applied Probability}, 31(6):2567--2584, 2021.

\bibitem{borda2021berry}
Bence Borda.
\newblock Berry--esseen smoothing inequality for the wasserstein metric on
  compact lie groups.
\newblock {\em Journal of Fourier Analysis and Applications}, 27(2):1--23,
  2021.

\bibitem{borda2021empirical}
Bence Borda.
\newblock Empirical measures and random walks on compact spaces in the
  quadratic wasserstein metric.
\newblock {\em arXiv preprint arXiv:2110.00295}, 2021.

\bibitem{borda2021equidistribution}
Bence Borda.
\newblock Equidistribution of random walks on compact groups ii. the
  wasserstein metric.
\newblock {\em Bernoulli}, 27(4):2598--2623, 2021.

\bibitem{caracciolo2014scaling}
Sergio Caracciolo, Carlo Lucibello, Giorgio Parisi, and Gabriele Sicuro.
\newblock Scaling hypothesis for the euclidean bipartite matching problem.
\newblock {\em Physical Review E}, 90(1):012118, 2014.

\bibitem{decreusefond1998fractional}
Laurent Decreusefond and Ali~Suleyman {\"U}st{\"u}nel.
\newblock Fractional brownian motion: theory and applications.
\newblock In {\em ESAIM: proceedings}, volume~5, pages 75--86. Citeseer, 1998.

\bibitem{dembo2004cover}
Amir Dembo, Yuval Peres, Jay Rosen, and Ofer Zeitouni.
\newblock Cover times for brownian motion and random walks in two dimensions.
\newblock {\em Annals of mathematics}, pages 433--464, 2004.

\bibitem{dieker2003spectral}
Antonius~Bernardus Dieker and Michael Mandjes.
\newblock On spectral simulation of fractional brownian motion.
\newblock {\em Probability in the Engineering and Informational Sciences},
  17(3):417--434, 2003.

\bibitem{fallahgoul2016fractional}
H.~Fallahgoul, S.~Focardi, and F.~Fabozzi.
\newblock {\em Fractional Calculus and Fractional Processes with Applications
  to Financial Economics: Theory and Application}.
\newblock Elsevier Science, 2016.

\bibitem{fournier2015rate}
Nicolas Fournier and Arnaud Guillin.
\newblock On the rate of convergence in wasserstein distance of the empirical
  measure.
\newblock {\em Probability Theory and Related Fields}, 162(3):707--738, 2015.

\bibitem{galeati2020prevalence}
Lucio Galeati and Massimiliano Gubinelli.
\newblock Prevalence of {$\rho$}-irregularity and related properties.
\newblock {\em arXiv preprint arXiv:2004.00872}, 2020.

\bibitem{goldman2020convergence}
Michael Goldman and Dario Trevisan.
\newblock Convergence of asymptotic costs for random euclidean matching
  problems.
\newblock {\em Probability and Mathematical Physics}, 2020.

\bibitem{graf2007foundations}
Siegfried Graf and Harald Luschgy.
\newblock {\em Foundations of quantization for probability distributions}.
\newblock Springer, 2007.

\bibitem{jalowy2021wasserstein}
Jonas Jalowy.
\newblock The wasserstein distance to the circular law.
\newblock {\em arXiv preprint arXiv:2111.03595}, 2021.

\bibitem{jumarie2006new}
Guy Jumarie.
\newblock New stochastic fractional models for malthusian growth, the
  poissonian birth process and optimal management of populations.
\newblock {\em Mathematical and computer modelling}, 44(3-4):231--254, 2006.

\bibitem{ledoux2017optimal}
Michel Ledoux.
\newblock On optimal matching of gaussian samples.
\newblock {\em Zap. Nauchn. Sem. POMI}, 457(0):226--264, 2017.

\bibitem{ledoux2019optimal}
Michel Ledoux and Jie-Xiang Zhu.
\newblock On optimal matching of gaussian samples iii.
\newblock {\em arXiv preprint arXiv:1911.07579}, 2019.

\bibitem{marton1996measure}
Katalin Marton.
\newblock A measure concentration inequality for contracting markov chains.
\newblock {\em Geometric {\&} Functional Analysis GAFA}, 6(3):556--571, 1996.

\bibitem{nourdin2012selected}
Ivan Nourdin.
\newblock {\em Selected aspects of fractional Brownian motion}, volume~4.
\newblock Springer, 2012.

\bibitem{riekert2021wasserstein}
Adrian Riekert.
\newblock Wasserstein convergence rate for empirical measures of markov chains.
\newblock {\em arXiv preprint arXiv:2101.06936}, 2021.

\bibitem{santambrogio2015optimal}
Filippo Santambrogio.
\newblock Optimal transport for applied mathematicians.
\newblock {\em Birk{\"a}user, NY}, 55(58-63):94, 2015.

\bibitem{steinerberger2021wasserstein}
Stefan Steinerberger.
\newblock Wasserstein distance, fourier series and applications.
\newblock {\em Monatshefte f{\"u}r Mathematik}, 194(2):305--338, 2021.

\bibitem{steinerberger2021green}
Stefan Steinerberger.
\newblock A wasserstein inequality and minimal green energy on compact
  manifolds.
\newblock {\em Journal of Functional Analysis}, 281(5):109076, 2021.

\bibitem{talagrand2018scaling}
Michel Talagrand.
\newblock Scaling and non-standard matching theorems.
\newblock {\em Comptes Rendus Mathematique}, 356(6):692--695, 2018.

\bibitem{villani2009optimal}
C{\'e}dric Villani.
\newblock {\em Optimal transport: old and new}, volume 338.
\newblock Springer, 2009.

\bibitem{wang2021convergence}
Feng-Yu Wang.
\newblock Convergence in wasserstein distance for empirical measures of
  semilinear spdes.
\newblock {\em arXiv preprint arXiv:2102.00361}, 2021.

\bibitem{wang2021precise}
Feng-Yu Wang.
\newblock Precise limit in wasserstein distance for conditional empirical
  measures of dirichlet diffusion processes.
\newblock {\em Journal of Functional Analysis}, 280(11):108998, 2021.

\bibitem{wang2022wasserstein}
Feng-Yu Wang.
\newblock Wasserstein convergence rate for empirical measures on noncompact
  manifolds.
\newblock {\em Stochastic Processes and their Applications}, 144:271--287,
  2022.

\bibitem{wang2019limit}
Feng-Yu Wang and Jie-Xiang Zhu.
\newblock Limit theorems in warsserstein distance for empirical measures of
  diffusion processes on riemannian manifolds.
\newblock {\em arXiv preprint arXiv:1906.03422}, 2019.

\bibitem{xiao2006properties}
Yimin Xiao.
\newblock Properties of local-nondeterminism of gaussian and stable random
  fields and their applications.
\newblock In {\em Annales de la Facult{\'e} des sciences de Toulouse:
  Math{\'e}matiques}, volume~15, pages 157--193, 2006.

\end{thebibliography}

\end{document}